\documentclass[11pt]{article}
\usepackage{graphicx,amsthm,fancyhdr,mathrsfs}
\usepackage{amsfonts}
\usepackage{amsmath}
\usepackage{multicol}
\usepackage{mathrsfs}
\usepackage{multirow}
\usepackage{amssymb}
\usepackage{url}
   \usepackage[colorlinks=true,citecolor=blue]{hyperref}
\usepackage{fancyhdr}
\usepackage{indentfirst}
\usepackage{enumerate}
\usepackage{amsthm}
\usepackage{color}
\usepackage{comment}
\usepackage{dsfont}
\usepackage{natbib}
\usepackage[misc]{ifsym}
\usepackage[toc,page]{appendix}
\usepackage{subfig}

\def\d{\mathrm{d}}
\def\laweq{\buildrel \d \over =}

\newcommand{\var}{\mathrm{Var}}
\newcommand{\cov}{\mathrm{Cov}}
\newcommand{\corr}{\mathrm{Corr}}

\newcommand{\E}{\mathbb{E}}
\newcommand{\R}{\mathbb{R}}
\newcommand{\I}{[0,1]}

\newcommand{\N}{\mathbb{N}}
\newcommand{\p}{\mathbb{P}}

\newcommand{\id}{\mathds{1}}

\newcommand{\be}{\mathbf{e}}
\newcommand{\bx}{\mathbf{x}}
\newcommand{\by}{\mathbf{y}}
\newcommand{\bz}{\mathbf{z}}
\newcommand{\bX}{\mathbf{X}}

\newcommand{\bu}{\mathbf{u}}

\newcommand{\bp}{\mathbf{p}}

\newcommand{\bzero}{\mathbf{0}}
\newcommand{\bone}{\mathbf{1}}

\newcommand{\T}{^{\top}}
\newcommand{\diag}{\operatorname{diag}}

\renewcommand{\(}{\left(}
\renewcommand{\)}{\right)}

\newcommand{\supp}{\mathrm{Supp}}

\renewcommand{\ge}{\geqslant}
\renewcommand{\le}{\leqslant}
\renewcommand{\geq}{\geqslant}
\renewcommand{\leq}{\leqslant}
\renewcommand{\epsilon}{\varepsilon}

\usepackage{array}
\newcommand{\PreserveBackslash}[1]{\let\temp=\\#1\let\\=\temp}
\newcolumntype{C}[1]{>{\PreserveBackslash\centering}p{#1}}
\newcolumntype{R}[1]{>{\PreserveBackslash\raggedleft}p{#1}}
\newcolumntype{L}[1]{>{\PreserveBackslash\raggedright}p{#1}}

\usepackage{booktabs,array}

\newcount\rowc

\makeatletter
\def\ttabular{%
\hbox\bgroup
\let\\\cr
\def\rulea{\ifnum\rowc=\@ne \hrule height 1.3pt \fi}
\def\ruleb{
\ifnum\rowc=1\hrule height 1.3pt \else
\ifnum\rowc=6\hrule height \heavyrulewidth
   \else \hrule height \lightrulewidth\fi\fi}
\valign\bgroup
\global\rowc\@ne
\rulea
\hbox to 10em{\strut \hfill##\hfill}%
\ruleb
&&%
\global\advance\rowc\@ne
\hbox to 10em{\strut\hfill##\hfill}%
\ruleb
\cr}
\def\endttabular{%
\crcr\egroup\egroup}

\theoremstyle{plain}
\newtheorem{theorem}{Theorem}
\newtheorem{corollary}{Corollary}
\newtheorem{lemma}{Lemma}
\newtheorem{proposition}{Proposition}
\theoremstyle{definition}
\newtheorem{definition}{Definition}
\newtheorem{example}{Example}

\newtheorem{remark}{Remark}

\usepackage{tikz}

\usepackage[compact]{titlesec}

\makeatletter
\DeclareRobustCommand{\bsquare}{%
  \mathop{\vphantom{\sum}\mathpalette\bigstar@\relax}\slimits@
}

\newcommand{\bigstar@}[2]{%
  \vcenter{%
    \sbox\z@{$#1\sum$}%
    \hbox{\resizebox{.9\dimexpr\ht\z@+\dp\z@}{!}{$\m@th\dsquare$}}%
  }%
}
\makeatother

\newcommand{\dsquare}{\mathop{  \square} \displaylimits}

\usepackage[onehalfspacing]{setspace}

\begin{document}

\title{Invariant correlation under marginal transforms}

\author{Takaaki Koike\thanks{
Graduate School of Economics, Hitotsubashi University, Japan.
Email:\texttt{takaaki.koike@r.hit-u.ac.jp}
} \and Liyuan Lin\thanks{Department of Econometrics and Business Statistics, Monash University, Australia.  Email: \texttt{liyuan.lin@monash.edu}
} \and  Ruodu Wang\thanks{Department of Statistics and Actuarial Science, University of Waterloo, Canada.
Email: \texttt{wang@uwaterloo.ca}
}
}
\maketitle

\begin{abstract}
A useful property of independent samples is that their correlation remains the same  after applying marginal  transforms.
This invariance property plays a fundamental role in statistical inference, but does not hold in general for dependent samples.
In this paper, we study this invariance property on the Pearson correlation coefficient and its applications.
A multivariate random vector is said to have an invariant correlation if its pairwise correlation coefficients remain unchanged under any common marginal transforms.
For a bivariate case, we characterize all models of such a random vector via a certain combination of comonotonicity---the strongest form of positive dependence---and independence.
In particular, we show that the class of exchangeable copulas with invariant correlation is precisely described by what we call positive Fr\'echet copulas.
In the general multivariate case, we characterize the set of all invariant correlation matrices via the clique partition polytope.
We also propose a positive regression dependent model that admits any prescribed invariant correlation matrix.
Finally, we show that all our characterization results of invariant correlation, except one special case, remain the same if the common marginal transforms are confined to the set of increasing ones.
\medskip
\\
\hspace{0mm}\\
\noindent \emph{MSC classification:}
60E05, 
62E15, 
62H10,  
62H20. 
\\
\noindent   \emph{Keywords:}
correlation coefficient;
dependence matrices;
Fr\'echet copula; 
positive regression dependence.\\
\end{abstract}

\section{Introduction}\label{sec:intro}

The  {(Pearson) correlation coefficient} is one of the most popular measures of quantifying the strength of statistical dependence.
It 
is also called the linear correlation coefficient, 
reflecting on the fact that it measures linear dependence among random variables. 
As such, the correlation coefficient is  preserved under common linear transforms.
When non-linear transforms are applied, the correlation coefficient typically changes except some specific dependence structures such as independence.
In fact, this invariance is a useful and notable property of independent samples, and plays a fundamental role in statistics and probability theory.

To define this invariance property more formally, a $d$-dimensional random vector $\mathbf{X}=(X_1,\dots,X_d)$ with $\operatorname{Var}(X_i)\in (0,\infty)$, $i\in[d]=\{1,\dots,d\}$, is said to have an invariant correlation matrix $R=(r_{ij})_{d\times d}$ if
\begin{equation}\label{eq:inv:corr:all}
\corr(X_i,X_j) = \corr(g(X_i),g(X_j))=r_{ij}
\end{equation}
for all measurable functions $g:\R \to \R$ such that $\corr(g(X_i),g(X_j))$ is well defined for all $i,j\in[d]$.
Although this property appears strong, it is satisfied by some specific models existing in the literature, such as the conformal p-value, a useful statistical tool in machine learning and non-parametric inference~\citep{VGS05}.
The recent work
\cite{BCLRS23} showed that~\eqref{eq:inv:corr:all} holds for a set of null conformal p-values to discuss a broad class of combination tests.

Motivated by this recent work, this paper addresses the following questions.
(i) How large is the class of models with this invariance property?
(ii) Whether and how is this property connected to other dependence concepts?
(iii) What if the invariance property is confined to a smaller class of transforms relevant for applications, such as the class of monotone transforms?

Since~\eqref{eq:inv:corr:all} is essentially a bivariate property, we first focus on the case when $d=2$. 
It is straightforward to check that a pair of random variables $(X,Y)$ has an invariant correlation $r\in [-1,1]$ if $X$ and $Y$ have an identical marginal $F$ and their joint distribution $H$ is given by
\begin{align}\label{eq:r:frechet}
H(x,y)=r\min(F(x),F(y)) + (1-r) F(x)F(y),~~~ x,y\in \R.
\end{align}
We say that this model or its distribution is {$r$-Fr\'echet} due to its connection to the Fr\'echet copulas; see~\cite{DDGK05,YCZ06} for applications in risk modeling.
If $r\in [0,1]$, this model is a probabilistic mixture of comonotonic and independent cases, where two random variables $X$ and $Y$ are said to be {comonotonic} if $X=f(Z)$ and $Y=h(Z)$ for some random variable $Z$ and two increasing functions $f,h:\R\rightarrow \R$.
All terms like ``increasing" and ``decreasing" are in the non-strict sense in this paper.

Since invariant correlation is a strong property, one may wonder whether models of the form~\eqref{eq:r:frechet}
are exhaustive. It turns out that this is not the case, but not far away from the truth. 
As one of our main contributions, 
we give a complete characterization of all bivariate random vectors having an invariant correlation~\eqref{eq:inv:corr:all} using the asymmetric analogs of independence and the $r$-Fr\'echet model, which we call {quasi-independence} (Definition~\ref{def:quasi-ind}) and the {quasi-$r$-Fr\'echet model} (Definition~\ref{def:quasi:frechet}), respectively.
Our characterization results are summarized in Table \ref{table:IC}.  
We prove that zero invariant correlation is equivalent to quasi-independence if neither of the marginals is bi-atomic (i.e., supported on two-points), and to independence otherwise.
We also find that, for non-identical marginals, non-zero invariant correlation is not possible except when both marginals are bi-atomic.
For the case of identical marginals, it turns out that a model has a non-zero invariant correlation $r$ if and only if it is quasi-$r$-Fr\'echet.
In particular, when $X$ and $Y$ are exchangeable (i.e., $(X,Y)\laweq (Y,X)$, where $\laweq$ stands for equality in distribution), then the   $r$-Fr\'echet model~\eqref{eq:r:frechet} characterizes the model $(X,Y)$ with invariant correlation $r$.
 \begin{table}
\caption{Summary of characterization results for bivariate distributions with invariant correlation, where QI stands for quasi-independence,  QF stands for the quasi-Fr\'echet model, IN stands for  independence, 
IS stands for  models with identical marginal supports,
and   $\varnothing$ means no such model.}\label{table:IC}
\def\arraystretch{1.3}
\vskip-0.3cm\hrule
\smallskip
\centering\small
\begin{tabular}{c c cc}
 \multicolumn{2}{c}{Marginals $F, G$}& $r=0$ &$r\neq 0$\\
 \hline
 \multicolumn{2}{l}{$F=G$} & QI (Theorem  \ref{thm:ch:zero:ic}/\ref{thm:main:sec:4}) & QF (Theorem \ref{thm:main:sec:4}) \\
 \hline
 \multirow{3}{*}{$F\neq G$}   & both are bi-atomic & \multirow{2}{*}{IN (Proposition \ref{pro:ic0:one_bi})} &  IS (Proposition \ref{pro:bi-atomic}) \\
\cline{2-2}
\cline{4-4} 
  & precisely one is bi-atomic &   &  \multirow{2}{*}{$\varnothing$  (Theorem \ref{thm:main:sec:3})} \\
 \cline{2-2}
 \cline{3-3}
 & both  are not bi-atomic & QI (Theorem \ref{thm:ch:zero:ic})   &   \\
\end{tabular}
\hrule
\end{table}

Such a complete characterization cannot be expected for a general multivariate case since~\eqref{eq:inv:corr:all} is a requirement on bivariate margins. For a general $d\ge 2$, we prove that the set of all invariant correlation matrices attained by random vectors with continuous marginals coincides with the {clique partition polytope}~\citep{GW90}.
Based on this characterization, we provide a numerical procedure to check whether a given matrix is admissible as an invariant correlation matrix.
For $k\in \mathbb{N}$ and iid observations $Y_1,\dots,Y_k$ from a continuous and strictly increasing distribution $F$, we then consider the model:
 \begin{align}\label{eq:model:intro}
 \mathbf X= \Gamma \mathbf Y  =( \mathbf Z_1^\top \mathbf Y,\dots,\mathbf Z_d^\top  \mathbf Y),\quad
\Gamma=( \mathbf Z_1,\dots, \mathbf Z_d)^\top,\quad  \mathbf Y=(Y_1,\dots,Y_k)^\top,
 \end{align}
with $\mathbf Z_i$, $i\in[d]$, being $\{0,1\}^k$-valued Bernoulli random vector such that $\mathbf Z_i^\top \mathbf Z_i=1$ and $\Gamma$ being independent of $\mathbf Y$. 
We show that the model~\eqref{eq:model:intro} has an identical marginal $F$ and an invariant correlation matrix; moreover, for $k\ge d$, this model accommodates any admissible invariant correlation matrix.
In addition, we show that this model has {positive regression dependence}~\citep{L66}, an important dependence concept for controlling p-values in the context of multiple testing~\citep{BY01}. 

In many applications, only increasing transforms of the random variables are relevant since non-increasing transforms do not preserve the joint distribution of the ranks of the variables.
In view of this, we also study a variant of invariant correlation where we require \eqref{eq:inv:corr:all} to hold only for increasing transforms instead of all transforms.
By definition, this requirement is weaker 
than the original invariant correlation property~\eqref{eq:inv:corr:all}. 
A natural question is whether this formulation allows for more models than those characterized by~\eqref{eq:inv:corr:all}.
It turns out that invariant correlation confined to increasing transforms does not accommodate more models except for the case when both random variables are bi-atomic.

The paper is organized as follows. 
In Section~\ref{sec:2}, we begin with introducing notation, definitions and basic properties of invariant correlation.
Section~\ref{sec:3} is then devoted to  a full characterization of invariant correlation in the bivariate setting.
Invariant correlation matrix is studied in Section~\ref{sec:matrix}, where a model admitting any prescribed invariant correlation matrix is proposed.
This section also explores positive regression dependence of this model.
Section~\ref{sec:ic:inc} provides the characterization of invariant correlation under increasing transforms.
Discussion and directions for future research are given in Section~\ref{sec:conclusion}.
Proofs and auxiliary results are deferred to Appendix.

\section{Definitions and basic properties}\label{sec:2}

In this section, we provide definitions and basic properties of invariant correlation and related models.
We fix an atomless probability space $(\Omega,\mathcal A,\p)$ throughout the paper.
Denote by $\mathcal L^2$ the set of all non-degenerate real-valued random variables with finite variance. 
For $d\in\N$, denote by ${\mathcal B}(\R^d)$ the Borel $\sigma$-algebra on $\R^d$.
We also write $\mathbf{0}_d=(0,\dots,0)^\top$ and $\mathbf{1}_d=(1,\dots,1)^\top$ for the $d$-dimensional vector of zeros and ones, respectively.
 For a random vector $\mathbf X$, denote by $F_{\mathbf X}$ its distribution function. 
 
For $X,Y \in \mathcal L^2$, the {Pearson correlation coefficient} of $(X,Y)$ is $$\corr(X,Y)=\frac{\cov(X,Y)}{\sqrt{\var(X)\var(Y))}}.$$ 
We are interested in the following invariance of correlation under certain transforms of $(X,Y)$.
In this context, we call a function $g:\R \rightarrow \R$ {admissible} (for $(X,Y)$, omitted if clear) if it is measurable and $\corr(g(X),g(Y))$ is well defined, i.e., $g(X),g(Y)\in \mathcal L^2$.
Strictly monotone functions with bounded derivatives are always admissible. 


\begin{definition}
Let $r\in [-1,1]$.
A bivariate random vector $(X,Y)$ is said to have an {invariant correlation} $r$ 
if
\begin{align}\label{eq:inv:corr}
\corr(X,Y) = \corr(g(X),g(Y))=r\quad \text{for all admissible functions $g$.}
\end{align}
\end{definition}

The set of all  bivariate random vectors with invariant correlation $r$ is denoted by $\mathrm{IC}_r$. Moreover, let $\mathrm{IC}=\bigcup_{r\in [-1,1]}\mathrm{IC}_r$, which is the set of all bivariate random vectors having any invariant correlation.

The concept of invariance correlation can be naturally formulated dimension larger than $2$.
For $d\ge 2$, a $d$-dimensional random vector $\bX$ is said to have an {invariant correlation matrix} $R=(r_{ij})_{d\times d}$ if $(X_i,X_j) \in \mathrm{IC}_{r_{ij}}$ for every pair $(i,j) \in [d]^2$. The set of all such random vectors is denoted by $\mathrm{IC}_{R}$.
For a distribution function $H$ on $\R^d$, we write $\mathbf X\sim H$ if $\mathbf X$ is distributed according to $H$, and  $H \sim \mathrm{IC}_R$ means $\mathbf X \in \mathrm{IC}_R$ for some $\mathbf X\sim H$. 
Denote by $\mathrm{IC}^{d}=\bigcup_{R \in \mathcal P_d}\mathrm{IC}_R$ the collection of all $d$-dimensional random vectors with invariant correlation matrix, where $\mathcal P_d$ is the set of all $d\times d$ correlation matrices. A measurable function $g$ is admissible  for $\mathbf X$ if $g(X_i)\in \mathcal L^2$ for each $i\in [d]$.
Note that $\mathrm{IC}^2=\mathrm{IC}$ is a special case.  
Although the problem of   invariance correlation is more general in dimension $d\ge 2$, 
its  theoretical challenge lies mostly in the case $d=2$, because the correlation coefficient is naturally a bivariate concept. 

Let us first look at a few basic examples of invariant correlation.

\begin{example}\label{ex:independent}
If a $d$-dimensional random vector $\mathbf X=(X_1, \dots, X_d)$ is independent, then so is $(g(X_1),\dots,g(X_d))$  for every measurable $g$. Hence, we have $\mathbf X \in \mathrm{IC}_{I_d}$, where $I_d=\operatorname{diag}(\mathbf{1}_d)$. In dimension $d=2$, we have $(X_1, X_2) \in \mathrm{IC}_0$.
\end{example}


\begin{example}\label{ex:frechet}
Suppose that $A$ is a random event with $\p(A)=r\in \I$, and let $\mathbf X=(X, \dots, X)$ and $\mathbf X^\bot=(X_1, \dots, X_d)$, where $X, X_1, \dots, X_d$ are iid random variables independent of $A$.
Let
\begin{equation}\label{eq:3a}
\mathbf Y = (Y_1,\dots,Y_d)= \mathbf X\id_A + \mathbf X^\bot \id_{A^c}.
\end{equation}
It is easy to verify that, for $i,j\in[d]$, $i\neq j$, and for every admissible $g$,
$$
\corr(g(Y_i),g(Y_j)) = \frac{\cov(g(Y_i),g(Y_j))}{\var(g(X))} = \frac{\var(g(X))\p(A)}{\var(g(X))}  = \p(A) =r.
$$
Hence, we have $\mathbf Y\in \mathrm{IC}^d$ where  all  off-diagonal elements in the invariant correlation matrix  equal $r$.
\end{example}

The following simple properties follow directly from the definition of invariant correlation, and will be used repeatedly in the subsequent analyses.


\begin{proposition}\label{lem:h}
If $\mathbf X=(X_1, \dots, X_d) \in \mathrm{IC}_R$ for $R\in \mathcal P_d$ and  $h$ is an admissible function for $\mathbf X$, then $(h(X_1), \dots, h(X_d))\in \mathrm{IC}_R$.
\end{proposition}

For a $d$-dimensional random vector $\mathbf X$, its copula $C$  is a distribution function on $[0,1]^d$ with standard uniform  marginals such that 
$
F_{\mathbf X}(x_1,\dots,x_n) = C(F_1(x_1),\dots,F_n(x_n))$ for $(x_1,\dots,x_n)\in \R^n$,
where   $F_1,\dots,F_n$ are the marginals of ${\mathbf X}$; see \cite{J14} for background on copulas. 
A copula of $\mathbf X$ is unique on $\operatorname{Ran}(F_1)\times \cdots \times \operatorname{Ran}(F_d)$, where $\operatorname{Ran}(F)$ is the range of a distribution funtion $F$.
In particular, if $\mathbf X$ has continuous marginals, then its copula is unique.  
It is known that every $d$-dimensional copula $C$ satisfies $W_d(\bu)\leq C(\bu)\leq M_d(\bu)$ for all $\bu=(u_1\dots,u_d) \in \I^d$, where $W_d(\bu)=\max(u_1+\cdots+u_d +1-d,0)$ and $M_d(\bu)=\min(u_1,\dots,u_d)$ are called {Fr\'echet bounds}.
Note that $M_d$ is a $d$-dimensional copula for every $d\geq 2$, and $W_d$ is a copula only when $d=2$.
These bounds, together with the {independence copula} $\Pi_d(\bu)=\prod_{i=1}^d u_i$, $\bu \in \I^d$, play fundamental roles in dependence analysis.

For $d=2$, we write $M=M_2$, $\Pi=\Pi_2$ and $W=W_2$ in short.
A (bivariate) {Fr\'echet copula} is a $2$-dimensional copula defined as
\begin{align}\label{eq:def:frechet:cop}
C_{r,s}^{\text{F}}(u,v)=r M(u,v) + s W(u,v) + (1-r-s) \Pi(u,v),\quad (u,v)\in \I^2,
\end{align}
which is parametrized by $r,s \in\I$ such that $r+s \leq 1$.
If $s=0$ in~\eqref{eq:def:frechet:cop}, a Fr\'echet copula is called {positive} and we denote it by $C_r^{\text{F}}$.
In fact, when $d=2$ and the identical distribution of $X,X_1,X_2$ in Example \ref{ex:frechet} is continuous, then $\mathbf{X}$ and $\mathbf{X}^\perp$ has the copula $M$ and $\Pi$, respectively.
Moreover, the copula of $\mathbf Y$ is the positive  Fr\'echet copula $C_r^{\text{F}}$.
It is easy to see that a bivariate random vector $(Y_1,Y_2)$ has a positive Fr\'echet copula $C_r^{\text{F}}$
if and only if the stochastic representation \eqref{eq:3a} with $d=2$ holds almost surely for some event $A$ such that $\p(A)=r$.

Let $\mathbf{X}$ be a random vector with continuous and identical marginals.
A useful consequence of Proposition~\ref{lem:h} is that it suffices to consider the copula of $\mathbf{X}$ to analyze invariant correlation of $\mathbf{X}$.

\begin{corollary}\label{cor:copula}
 Suppose that $\mathbf X\sim H$ has identical continuous and strictly increasing marginals, a correlation matrix $R\in \mathcal P_d$ and a copula $C$.
 Then $H\sim  \mathrm{IC}_R$ if and only if $ C\sim \mathrm{IC}_R$.
\end{corollary}

We next introduce some special structures on bivariate distributions, which turn out to be equivalent to invariant correlation in Sections~\ref{sec:3} and~\ref{sec:ic:inc}.
We first present a concept of dependence which is close to independence.
\begin{definition}\label{def:quasi-ind}
We say that a random vector $(X,Y)\sim H$ with marginals $F$ and $G$ is {quasi-independent} 
if   \begin{equation}\label{eq:quasi}
 \frac{H(x,y)+ H(y,x)}{2}= \frac{1}{2}F(x)G(y)+ \frac{1}{2}F(y) G(x) \mbox{~~for all $x,y\in \R$}.
  \end{equation}
\end{definition}
By a standard probabilistic argument,~\eqref{eq:quasi} is equivalent to 
\begin{equation}\label{eq:quasi:2}
  \p(X \in A, Y\in B)+\p(X \in B, Y\in A)=\p(X\in A)\p(Y \in B)+\p(X\in B)\p(Y\in A),
  \end{equation}  for all $A, B \in {\mathcal B}(\R)$; see Lemma \ref{lem:quasi} in Appendix \ref{app:2}.
  Clearly, independence implies quasi-independence, but the opposite implication is not true. Quasi-independence plays an important role in characterizing the dependence structure for $(X,Y) \in \mathrm{IC}_0$.



We then consider an extension of the positive Fr\'echet copula so that the marginals are not standard uniform and the dependence structure is not exchangeable.
Note, however, that the marginal distributions are assumed to be identical in the following concept.

\begin{definition}\label{def:quasi:frechet}
 Let $(X,Y)$ be a random vector with identical marginal $F$ and joint distribution $H$. We say that $(X,Y)$ is {quasi-$r$-Fr\'echet} for some $r\in [-1,1]$ if  
\begin{align}\label{eq:identical:identity}
\frac{H(x,y)+H(y,x)}{2}=r \min(F(x),F(y)) + (1-r)F(x)F(y),~~~\mbox{for all $x,y\in \R$}.
\end{align}
\end{definition}

The need to introduce quasi-independence and quasi-Fr\'echet model stems from the permutation invariance of the correlation coefficient $\operatorname{Corr}(X,Y)=\operatorname{Corr}(Y,X)$, that is, correlation coefficient does not capture directional (non-exchangeable) dependence.
Note that we require identical marginals of $(X,Y)$ in the quasi-Fr\'echet model, but no such requirement is imposed on quasi-independence. When $(X,Y)$ has an identical marginal, the quasi-Fr\'echet model with $r=0$ reduces to quasi-independence. 

Denote by $\mu$, $\mu_{+}$ and $\mu_{\perp}$ the probability measures induced by the cumulative distribution functions $H$,
$H_{+}$ and $H_{\perp}$, respectively, where $H_{+}(x,y)=\min(F(x),F(y))$ and $H_{\perp}(x,y)=F(x)F(y)$, $(x,y)\in \R$.
By the standard measure-theoretic argument (see Lemma~\ref{lem:quasi_Fre} in Appendix \ref{app:2}), 
we have that the condition \eqref{eq:identical:identity} is equivalent to 
 \begin{align} \label{eq:identical:identity:set}
\frac{ \mu(A\times B) + \mu(B \times A)}{2}=r \mu_{+}(A\times B) + (1-r)\mu_{\perp}(A\times B),\quad \text{ for all }A,B\in {\mathcal B}(\R).
\end{align}
The term ``quasi-Fr\'echet" reflects its connection to  the Fr\'echet copula  in \eqref{eq:def:frechet:cop}, although here $r$ may be negative.
When $(X,Y)$ is exchangeable, the quasi-$r$-Fr\'echet model reduces to~\eqref{eq:r:frechet},
which again yields the positive Fr\'echet copula $C_r^\text{F}$ when the identical marginal is the standard uniform. 

\section{Bivariate invariant correlation}\label{sec:3}

In this section we focus on the bivariate case $d=2$.
We will characterize dependence structures of bivariate random vectors $(X,Y)$ with invariant correlation.
The main results of this section are summarized in Table \ref{table:IC}.  
 For a random variable $X$, the support of the distribution of $X$ is given by $\supp(X)=\{x\in \R: \p(x-\epsilon< X\le x+\epsilon)>0 ~\mbox{for all}~~\epsilon>0\}$.

\subsection{Zero invariant correlation}\label{sec:ic0}
We first characterize the joint distribution in the case $(X,Y)$ has an invariant correlation $0$.
As we will see later, this case is fundamentally different from the case $r\ne 0$.

In Example \ref{ex:independent}, we have seen that $(X,Y) \in \mathrm{IC}_0$ if $(X,Y)$ is independent.
A natural first question is whether independence is the only possibility for $\mathrm{IC}_0$.

\begin{theorem}\label{thm:ch:zero:ic}
 For   $X,Y \in \mathcal L^2$,   $(X,Y) \in \mathrm{IC}_0$ if and only if $(X,Y)$ is quasi-independent.
\end{theorem}

Theorem \ref{thm:ch:zero:ic} shows the special role played by the notion of quasi-independence  as the only structure for  zero invariant correlation.
Let $(\pi_1,\pi_2)$ be a random vector uniformly distributed on $\{(1,2), (2,1)\}$.
For a random vector $(X_1, X_2)$ independent of $(\pi_1,\pi_2)$, we call $(X_{\pi_1}, X_{\pi_2})$ the {random rearrangement} of $(X_1, X_2)$, whose distribution is the equal mixture of $(X_1,X_2)$ and $(X_2,X_1)$.
 Although quasi-independence does not imply independence, the following proposition shows that quasi-independence of $(X_1,X_2)$ implies independence of its random rearrangement if $X_1$ and $X_2$ have the same distribution.
 

\begin{proposition}\label{pro:ic0:independent}
 Suppose that $X_1, X_2 \in \mathcal L^2$ have the same distribution. Then $(X_1,X_2) \in \mathrm{IC}_0$ if and only if its random rearrangement is independent.
\end{proposition}

In particular, if $(X_1,X_2)$ is exchangeable, that is, $(X_{2}, X_{1}) \laweq (X_1,X_2)$, then zero invariant correlation is equivalent to independence.

Let $\vert A \vert$ be the cardinality of a set $A \subseteq \R$; $|A|=\infty$ if $A$ is infinite. 
A random variable $X$ is {$n$-atomic} if  
$\vert\supp(X)\vert= n$;
that is, its distribution is supported on  $n$ distinct points. 
We use the term ``bi-atomic" in case $n=2$ and ``tri-atomic" in case $n=3$. 
These terms are applied to both  random variables and their distributions.
Theorem \ref{thm:ch:zero:ic} identifies the joint distribution of $(X,Y) \in \mathrm{IC}_0$ when $X$ and $Y$ are atomic random variables.

\begin{proposition}\label{cor:n:atomic}
Suppose $X$ is $m$-atomic and $Y$ is $n$-atomic with $2\le m\le n$.  Let $P=(p_{ij})_{n\times n}$ with $p_{ij}=\p(X=x_i,Y=y_j)$ for $i\in[m]$, $j\in [n]$ and $p_{ij}=0$ for $m<i\le n$, $j\in [n]$, $\mathbf p = P\bone_n$ and $\mathbf q = P\T\bone_n$.
Then the followings are equivalent.
\begin{enumerate}[(i)]
\item $(X,Y)\in \mathrm{IC}_0$.
\item $P+P^\top=\mathbf p\mathbf q^\top +\mathbf q\mathbf p^\top$.
\item $P= \mathbf p \mathbf q^\top + S$,
where $ S =(s_{ij})_{n\times n}$ satisfies:
(a) $p_{i}q_{j}+s_{ij}\ge 0$ for each $(i,j)$;
(b) each of its row sums and column sums is $0$; and
(c) $s_{ij}=-s_{ji}$ for each $(i,j)$, i.e., $S$ is anti-symmetric.
\end{enumerate}
\end{proposition}

For the special case $m=n=3$ and $x_i=y_i$ for $i\in[3]$, we have an explicit representation for the probability matrix $P$ in Example \ref{cor:tri0}.

\begin{example}\label{cor:tri0}
 Assume that $X$ and $Y$ are tri-atomic random variables with identical supports. Then,  $(X,Y)\in \mathrm{IC}_0$ if and only if the probability matrix is given by $P=\mathbf p\mathbf q^\top+S$, where $S=(s_{ij})_{3\times 3}$ is an anti-symmetric matrix such that $s_{ii}=0$ for $i\in [3]$ and $s_{12}=s_{23}=s_{31}=\epsilon
 \in [-\min(p_1q_2,p_2q_3,p_3q_1), \min (p_1q_3,p_2q_1,p_3q_2)]$.
\end{example}

By taking $p_1\downarrow0$, we have  $\epsilon\to 0$ in Example \ref{cor:tri0}. Hence, it is expected that zero invariant correlation is characterized by independence if one of the marginals is bi-atomic.
This is formally stated in the next proposition. 

\begin{proposition}\label{pro:ic0:one_bi}
Let $X,Y\in \mathcal L^2$ where $X$ is bi-atomic. 
If $(X,Y)$ is quasi-independent, then $(X,Y)$ is independent.
In particular, 
$(X,Y) \in \mathrm{IC}_0$ if and only if $(X,Y)$ is independent.
\end{proposition}

At this point, we have fully characterized zero invariant correlation and its relation to independence. Next, we consider general invariant correlation $r$ for $r\in [-1,1]$.

\subsection{Invariant correlation with non-identical marginals}\label{sec:non:identical}

The main message of this section is that, for $X$ and $Y$ with different marginals,  having a non-zero invariant correlation 
is   impossible, except for the very special case of bi-atomic distributions.


First, suppose that $X$ and $Y$ are {bi-atomic} random variables.
The following result shows that $(X,Y) \in \mathrm{IC}$ if $X$ and $Y$ have the same support. In this case, the invariant correlation can   be any number in $[-1,1]$. In particular, we have that $\corr(X,Y)=1$ implies comonotonicity, $\corr(X,Y)=0$ implies independence, and $\corr(X,Y)=-1$ implies counter-monotonicity (i.e., $X$ and $-Y$ are comonotonic).
However, if $X$ and $Y$ have different supports, then independence is the only possible case of invariant correlation.

 \begin{proposition}\label{pro:bi-atomic}
  Let $X$ and $Y$ be bi-atomic random variables.
  Then $(X,Y) \in \mathrm{IC}$  if and only if either (i) $X$ and $Y$ have the same support, or (ii) $(X,Y)$ is independent.
 \end{proposition}

Note that case (ii) in Proposition~\ref{pro:bi-atomic} is equivalent to zero invariant correlation.
Therefore, for a pair of bi-atomic random variables, $\mathrm{IC}_r$ for $r\neq 0$ is characterized as the set of all models with identical marginal supports.

Next, we obtain the following observations by carefully investigating tri-atomic models; see Lemma~\ref{pro:tri} in~Appendix \ref{app:3.2} for details.

 \begin{enumerate}[(i)]
 \item For bi-atomic $X$ and tri-atomic $Y$ such that $\supp(X)\subsetneq\supp(Y)$, we have $(X,Y) \in \mathrm{IC}$ if and only if $(X,Y) \in \mathrm{IC}_0$.
 \item Let $X$ and $Y$ be tri-atomic random variables with the same support. 
 If $(X,Y) \in \mathrm{IC}$, then   $F_X=F_Y$ or  $\corr(X,Y)=0$.  
 \end{enumerate}
If $X$ and $Y$ are general random variables with different distributions, then we can take $h$ in Proposition~\ref{lem:h} to transform $X$ and $Y$ into bi-atomic or tri-atomic random variables, so that we can apply Proposition \ref{pro:bi-atomic} and the above observations (i) and (ii). 
This argument leads to the main result of this section.  

 \begin{theorem}\label{thm:main:sec:3}
Suppose that $X ,Y \in \mathcal L^2$ have different distributions and $|\supp(Y)|>2$.
Then
 $(X,Y) \in \mathrm{IC}$  if and only 
 $(X,Y) \in \mathrm{IC}_0$ (which is characterized in Theorem \ref{thm:ch:zero:ic}).
\end{theorem}

Theorem \ref{thm:main:sec:3} suggests that the case of $\mathrm{IC}_r$, $r\ne 0$ is very limited. 
In particular, the only possibility for having different marginals in $(X,Y)\in \mathrm{IC}_r$, $r\ne 0$, is that $X$ and $Y$ are both bi-atomic.   
 Together with Proposition~\ref{pro:bi-atomic}, Theorem \ref{thm:main:sec:3} characterizes all $(X,Y)\in \mathrm{IC}$ with different marginals. 
 

\subsection{Invariant correlation with identical marginals}\label{sec:identical}

In this section we explore the remaining case of $(X,Y) \in \mathrm{IC}$ where $F_X=F_Y$.
We start with $n$-atomic random variables $X$ and $Y$. Assume that $X$ and $Y$ take values in $\mathcal X=\{x_1,\dots,x_n\}$, $n\in \N$, with $x_1<\cdots<x_n$. Define the $n\times n$ probability matrix $P=(p_{ij})_{n \times n}$ of $(X,Y)$ as $p_{ij}=\p(X=x_i,Y=x_j)$.
Let $\bp=P\bone_n$ be the marginal probability of $X$ and $Y$.
Let $D=\diag(\bp)$.
For $\bx=(x_1,\dots,x_n)^\top$, we have
$
\E[X]=\E[Y]=\bx\T \bp$,
$\E[X^2]=\E[Y^2]=\bx\T D \bx$,
$\E[XY]= \bx\T P\bx$, and thus the Pearson correlation coefficient of $(X,Y)$ is given by
$$
\corr(X,Y)=\frac{\bx\T P\bx - \bx\T \bp \bp\T \bx}{
\sqrt{\bx\T D\bx  - \bx\T \bp\bp\T\bx}
\sqrt{\bx\T D \by  - \bx\T \bp \bp\T\bx}}=\frac{\bx\T(P-\bp\bp\T)\bx}{\bx\T(D-\bp\bp\T)\bx}.
$$



\begin{proposition}\label{prop:ic:finite:identical}
Let $X$ and $Y$ be identically distributed $n$-atomic random variables.
Then $(X,Y)\in \mathrm{IC}_r$ for some $r \in[-1,1]$ if and only if $(X,Y)$ is quasi-$r$-Fr\'echet, that is, the probability matrix $P$ satisfies
\begin{align}\label{eq:ic:finite:mixture:m:pi}
\frac{P +P\T}{2} =  r D + (1-r)\bp\bp\T.
\end{align}
Furthermore, the invariant correlation $r$ satisfies
\begin{align}\label{eq:range:invariant:correlation}
\underline r_\bp\leq r \leq 1,\quad\text{where}\quad
\underline r_\bp = \max\left(
\max_{j\in [n]}\left(-\frac{p_j}{1-p_j}\right),
\max_{ i,j \in [n],i\neq j}
\left(1-\frac{1}{p_ip_j}\right)
\right).
\end{align}
\end{proposition}
If $(X,Y)$  in Proposition \ref{prop:ic:finite:identical} is exchangeable, then $P\T=P$ in \eqref{eq:ic:finite:mixture:m:pi}, which yields  $P=  r D + (1-r)\bp\bp\T$. That is, $P$ is $r$-Fr\'echet as defined in \eqref{eq:r:frechet}.
\begin{remark}\label{rem:range:r}
We show that $-1/(n-1) \leq \underline r_\bp < 0$ for any $\bp$.
The upper bound for $\underline r_\bp$  is trivial since all terms in $\underline r_\bp$ are negative.
Since
\begin{align}\label{eq:rp:term}
\max_{j\in [n]}\left(-\frac{p_j}{1-p_j}\right) = - \left(\frac{1}{\min(p_1,\dots,p_n)}-1\right)^{-1},
\end{align}
we have $\underline r_\bp \uparrow 0$ if $\min(p_1,\dots,p_n) \downarrow 0$.
To show the lower bound for $\underline r_\bp$, notice that the minimum of~\eqref{eq:rp:term} is $-1/(n-1)$ when $\min(p_1,\dots,p_n)=1/n$, and this is attained if and only if $p_i=1/n$, $i\in [n]$.
Consequently, we have the lower bound $-1/(n-1)$.
Note that, for $\bp\neq (1/n)\mathbf{1}_n$,~\eqref{eq:rp:term} is strictly greater than $-1/(n-1)$ and thus $\underline r_\bp>-1/(n-1)$.
\end{remark}





We next present the main result of this section, which covers the case when $X$ and $Y$ have identical marginals. 

\begin{theorem}\label{thm:main:sec:4}
Let $X, Y\in \mathcal L^2$ be identically distributed.
Then $(X,Y) \in \mathrm{IC}_r$ if and only if $(X,Y)$ is quasi-$r$-Fr\'echet.
\end{theorem}





Theorem \ref{thm:main:sec:4} highlights the special role played by  quasi-Fr\'echet dependence. 
As seen in Proposition \ref{prop:ic:finite:identical} and Remark~\ref{rem:range:r}, the invariant correlation $r$
can be negative when the support of $F$ is finite.
On the other hand, we will see that $r\ge 0$ if $X\sim F$ is continuous or the support of $X$ is countable but infinite.

Next, we consider the case when the support of $X$ and $Y$ is discrete, that is, $\mathcal X=\{x_i: i \in \mathcal I\}$, $\mathcal{I} \subseteq \N$.
Write $\p(X=x_i,Y=x_j)=p_{ij}>0$, $i ,\,j\in \mathcal I$, and $\p(X=x_i)=p_i>0$, $i\in \mathcal I$.
In this special case, Theorem~\ref{thm:main:sec:4} can be stated in the following form, which is more general than  Proposition \ref{prop:ic:finite:identical} since $|\mathcal X|=\infty$ is allowed.
For identically distributed discrete random variables $X$ and $Y$, we have
 $(X,Y)\in\mathrm{IC}_r$ for some $r\in [-1,1]$ if and only if
\begin{align}\label{eq:ic:infinite:mixture:m:pi}
\frac{p_{ij} +p_{ji}}{2} =  r p_i \id_{\{i=j\}} + (1-r)p_i\,p_j\quad \text{for all }i,j \in{\mathcal I},
\end{align} 
with $\underline r_\bp \leq r \leq 1$, where
\begin{align*}
\underline r_\bp = \max\left\{
\sup_{j\in \mathcal I}\left(-\frac{p_j}{1-p_j}\right),
\sup_{ i,j \in \mathcal I,i\neq j}
\left(1-\frac{1}{p_ip_j}\right)
\right\}.
\end{align*}
In particular, if  $|\mathcal X|=\infty$, we have $\inf_{i \in \mathcal I}p_i=0$, and hence $\underline r_{\bp}=0$ and $0\le r \le 1$.

Next, we consider the case when $F_X=F_Y=F$ for a continuous and strictly increasing distribution $F$.
Let $(U,V)=(F(X),F(Y))\sim C$ be the unique copula of $(X,Y)$.
By using Corollary~\ref{cor:copula}, Theorem~\ref{thm:main:sec:4} can be stated as follows.

\begin{corollary}\label{cor:ic:cont:identical}
Suppose that $X$ and $Y$ have identical continuous and strictly increasing marginals. 
Then the following are equivalent:
\begin{enumerate}[(i)]
\item $(X,Y)\in \mathrm{IC}_r$;
\item the copula $C$ of $(X,Y)$ satisfies\begin{align}\label{eq:ic:cont:mixture:m:pi}
\frac{C(u,v)+C(v,u)}{2}=r M(u,v) + (1-r) \Pi(u,v)\quad\mbox{for all $(u,v)\in\I^2$};
\end{align}
\item $r \in[0,1]$ and the positive Fr\'echet copula $C_r^{\operatorname{F}}$ is the copula of the random rearrangement of $(X,Y)$.
\end{enumerate}
\end{corollary}

For any given exchangeable copula $D$, we can generalized  \eqref{eq:ic:cont:mixture:m:pi}   to the following property of  a copula $C$: 
$(C(u,v)+C(v,u))/2=D(u,v)$ for $(u,v)\in [0,1]^2$.
In this way, $C$ has similar properties to $D$, but is not necessarily exchangeable. This idea underlines   quasi-independence with $D=\Pi$ and quasi-$r$-Fr\'echet copulas with $D=rM+(1-r)\Pi$.

Unlike the case of atomic random variables, negative invariant correlation is not possible when the support of $X$ and $Y$ is uncountable. 
By Corollary~\ref{cor:ic:cont:identical}, the positive Fr\'echet copula $C_r^{\operatorname{F}}$ is characterized as an exchangeable copula with invariant correlation $r$.
Non-exchangeable copulas satisfying~\eqref{eq:ic:cont:mixture:m:pi} are constructed in the next example.

\begin{example}
For $r \in \I$, let $P$ be a $3\times 3$ stochastic matrix such that $(P+P^\top)/2=(1/9)\bone_3\bone_3\T$.
Let $C_P$ be the checkerboard copula (see Section~4.1.1~of \cite{DS16}) associated with $P$ with the interior of each grid filled by the independence copula.
Let $C=rM + (1-r)C_P$.
Then it is straightforward to check that the copula $C$ satisfies~\eqref{eq:ic:cont:mixture:m:pi}.
On the other hand, $C$ may not be a Fr{\'e}chet copula since $C_P$ is not exchangeable when $P$ is not symmetric.
\end{example}

\section{Invariant correlation  matrices}\label{sec:matrix}

In this section, we study   invariant correlation matrices as the multivariate extension of invariant correlation. 

\subsection{Characterization of invariant correlation matrices}\label{sec:cha:ICM}

Recall that a $d$-dimensional random vector $\bX$ has an {invariant correlation matrix} $R=(r_{ij})_{d\times d}$
if every pair $(i,j)$ of its components has an invariant correlation $r_{ij}$ for $i,j\in [d]$.
We confine ourselves to the case when $\bX$ is continuous and its marginals are identical. 
Denote by $\Theta_d$ the set of all possible $d \times d$ invariant correlation matrices of continuous random vectors with identical marginals. Our goal 
is to understand the set  $\Theta_d$ and its corresponding models. 

For continuous marginals, Corollary~\ref{cor:ic:cont:identical} immediately leads to the following result.
For a $d$-copula $C$ and $i,j \in [d]$, $i\neq j$, its $(i,j)$th marginal is denoted by $C_{ij}$.
\begin{proposition} \label{prop:high-d}
For  a continuous $d$-dimensional random vector $\bX$ with identical marginals,
$\bX\in \mathrm{IC}^d$ holds if and only if
for every $i,j \in [d]$ with $i \neq j$, there exists $r_{ij}\in [0,1]$ such that the copula $C_{ij}$ of $(X_i,X_j)$ satisfies
$(C_{ij}+C_{ji})/2 = r_{ij} M + (1-r_{ij})\Pi$.
  \end{proposition}

  
Let $k\in \mathbb N$.
For $i \in [d]$, let $\mathbf Z_i =(Z_{i 1},\dots,Z_{i k})^{\top}$ be a $\{0,1\}^k$-valued Bernoulli random vector with $\mathbf Z_i^\top \mathbf Z_i=1$. In other words, $\mathbf Z_i$ has a categorical distribution with $\sum_{j=1}^k Z_{ij}=1$. Write the matrix $\Gamma=(Z_{ij})_{d \times k}$, and call it a $d\times k$ {categorical random matrix}.  Let $\mathbf U$ be an independent uniform random vector on $[0,1]^k$ independent of $\Gamma$.
Let  
 \begin{align}\label{eq:newmodel}
 \mathbf X=\Gamma \mathbf U  =( \mathbf Z_1^\top \mathbf U,\dots,\mathbf Z_d^\top  \mathbf U).
 \end{align}
 We first show that $\mathbf X$ has an invariant correlation matrix.
 \begin{proposition}\label{prop:newmodel}
 For $\mathbf X$ in \eqref{eq:newmodel}, the following statements hold:
 \begin{enumerate}[(i)]
 \item $\mathbf X$ has standard uniform marginals; 
 \item $\mathbf X$ has an invariant correlation matrix;
 \item The correlation matrix of  $\mathbf X$ (equal to its tail-dependence matrix) is given by $\E[\Gamma  \Gamma^\top]$.
 \end{enumerate}
 \end{proposition}

 \begin{example}\label{ex:bates:model}
  Take $k=d+1$ and let $Z_{ij}=0$ for all $j\in [d]\setminus \{i\}$. This implies $Z_{ik}=1-Z_{ii}$ for $i\in [d]$. For this model, we have 
 $
 X_i =Z_{ii} U_i +   Z_{ik} U_{d+1} = (1-Z_{ik})U_i +   Z_{ik} U_{d+1},
 $ $i\in [d]$,
and the correlation coefficient between $X_i$ and $X_j$ is given by $\E[Z_{ik}Z_{jk}]$.
 \end{example}

\begin{remark}\label{rem:newmodel:general}
{Let $\mathbf{Y}=(Y_1,\dots,Y_k)^\top$ be iid observations from a  distribution $F$, and let $\Gamma$ be as in~\eqref{eq:newmodel} independent of $\mathbf{Y}$.
Following the proof of Proposition~\ref{prop:newmodel}, it can be shown that the model
     \begin{align}\label{eq:newmodel:general}
\mathbf X=\Gamma \mathbf Y  =( \mathbf Z_1^\top \mathbf Y,\dots,\mathbf Z_d^\top  \mathbf Y)
 \end{align}
 has an identical marginal distribution $F$ and an invariant correlation matrix $\mathbb{E}[\Gamma\Gamma^\top]$.
 The model~\eqref{eq:newmodel:general} has singular components, and can be useful in the following contexts, although we do not explore further in this paper.
 \begin{enumerate}[(i)]
     \item {Bootstrap samples}: We can regard $\mathbf{X}$ as a bootstrap sample drawn from $\mathbf{Y}$, where $\Gamma$ represents a resampling rule.
    \item {Sample duplication}: Regarding $F$ as a distribution of a statistic in a population, $\mathbf{X}$ can be seen as a sample from $F$ in the presence of {sample duplication}, that is, some individuals are extracted multiple times.
    In this case, $\Gamma$ provides as random duplication structure.
 \item {Shock model}:  Let $\mathbf{X}$ be a vector of (possibly curtate) failure times of identical components in a system. Then~\eqref{eq:newmodel:general} models simultaneous failures caused by common shocks represented by $\Gamma$. 
\item {Markov chain}: For $p_1,\dots,p_{d-1}\in [0,1]$, define a Markov chain as follows:
\begin{align*}
    X_1=Y_1\quad\text{and}\quad X_i=\begin{cases}
        X_{i-1},& \text{with probability $p_{i-1}$},\\
        Y_i, & \text{otherwise},\\
    \end{cases}\quad \text{$i\in \{2,\dots,d\}$,}
\end{align*}
where $Y_1,\dots,Y_d$ are iid from a distribution $F$.
Then $\mathbf{X}$ admits the representation~\eqref{eq:newmodel:general} by sequentially defining $\mathbf{Z}_i$, $i\in [d]$, as follows:
\begin{align*}
    \mathbf{Z}_1 = \be_1\quad\text{and}\quad
    \mathbf{Z}_i = 
    \begin{cases}    
    \mathbf{Z}_{i-1},& \text{with probability $p_{i-1}$},\\
    \be_i,&\text{otherwise},
    \end{cases}\quad \text{$i\in \{2,\dots,d\}$,}
\end{align*}
where $\be_1,\dots,\be_d$ is the standard basis of $\mathbb{R}^d$. This model describes a  Markov process, of which each step has a certain probability to stay in the previous state and otherwise randomly jumps to a state according to $F$.
\end{enumerate}
}
\end{remark}
 
 Denote by $\mathcal Z_{d,k}$ the set of all matrices of the form $\E[\Gamma \Gamma^\top]$ for some $d\times k$ categorical random matrix $\Gamma$.  Proposition \ref{prop:newmodel} indicates that $\mathcal Z_{d,k}$ is a subset of $\Theta_d$. 
Lemma \ref{prop:compatibility:1} in Appendix \ref{sec:proof:cha:ICM} shows that $\mathcal Z_{d,k}$ is the convex hull of the collection of all clique partition points~\citep{GW90, FSS17}, which we will explain below. 
 Let  $\mathbf{A}=(A_1,\dots,A_{k})$ be a $k$-partition of $[d]$, where some of $A_s$, $s\in[k]$, may be empty.
For $k\ge d$, the number of non-empty partition components in $\mathbf{A}$ is at most $d$.
The {clique partition point}   $\Sigma^{\mathbf{A}} =
(\Sigma_{ij}^\mathbf{A})_{d\times d}$
for the $k$-partition $\mathbf{A}=(A_1,\dots,A_k)$ of $[d]$ is  defined as 
\begin{align}\label{eq:cp}
\Sigma_{ij}^{\mathbf{A}}=\sum_{s=1}^k\id_{\{i,j\in A_s\}},~~~i,j\in [d].
\end{align}
In other words, $\Sigma^{\mathbf{A}}$ can be decomposed into disjoint  submatrices of all entries $1$ and the other entries are $0$.
We can verify that $\Sigma^{\mathbf{A}}\in  \mathcal Z_{d,k}$,
as it is the correlation matrix of the following model of the form~\eqref{eq:newmodel}:
\begin{align}\label{eq:model:cpp}
X_i^{\mathbf{A}}=\sum_{s=1}^k\id_{\{i\in A_s\}}U_s,  ~~~i\in [d]. 
\end{align}

Denote by $\mathcal S_{d,k}$ the collection of all possible clique partition points~\eqref{eq:cp} for all possible $k$-partitions of $[d]$, and  let $\mathrm{Conv}(\mathcal S_{d,k})$ be the convex hull of $\mathcal S_{d,k}$.

    \begin{theorem}\label{thm:ch:mat}
For $d\ge 2$,  $\Theta_d=\mathcal Z_{d,d}=\bigcup_{k\in \mathbb N} \mathcal Z_{d,k}  =  \mathrm{Conv}(\mathcal S_{d,d})$.
  \end{theorem}

 Theorem~\ref{thm:ch:mat} characterizes $\Theta_d$ as the convex hull of $\mathcal S_{d,d}$. 
By Theorem~\ref{thm:ch:mat}, it suffices to take $k=d$ to characterize $\Theta_d$. This is due to the fact that $\mathcal Z_{d,k}$ is increasing in $k$ and $\mathcal Z_{d,k}=\mathcal Z_{d,d}$ for all $k\ge d$; see Section~\ref{sec:proof:cha:ICM} for details.
The set $\mathrm{Conv}(\mathcal S_{d,d})$ is called the {clique partition polytope} in dimension $d$, which is known to be relevant to matrix compatibility of other dependence measures; see~Appendix~\ref{app:other:measures} for details.
  This relevance also clues in the volume of $\Theta_d$ in comparison with the set of all correlation matrices.

\begin{example}
    {
    Let $d=3$.
    Then the set $\mathcal S_{3,3}$ consists of the following matrices:
    \begin{align*}
    \begin{pmatrix}
    1 &1 & 1\\
       1 &1 & 1\\
       1 &1 & 1\\
\end{pmatrix},\quad
\begin{pmatrix}
    1 &1 & 0\\
       1 &1 & 0\\
       0 &0 & 1\\
\end{pmatrix},\quad
\begin{pmatrix}
    1 &0 & 1\\
       0 &1 & 0\\
       1 &0 & 1\\
\end{pmatrix},\quad
\begin{pmatrix}
    1 &0 & 0\\
       0 &1 & 1\\
       0 &1 & 1\\
\end{pmatrix},\quad
\begin{pmatrix}
    1 &0 & 0\\
       0 &1 & 0\\
       0 &0 & 1\\
\end{pmatrix},
\end{align*}
which are generated by the partitions 
$(\{1,2,3\})$, $(\{1,2\},\{3\})$, $(\{1,3\},\{2\})$, $(\{2,3\},\{1\})$ and
$(\{1\},\{2\},\{3\})$, respectively.
Hence, when correlation matrices are embedded on $\mathbb{R}^3$ by $(r_{ij})_{3\times 3}\mapsto (r_{12},r_{13},r_{23})$, the set $\Theta_3=\mathrm{Conv}(\mathcal S_{3,3})$ forms a {triangular bipyramid}, a hexahedron with six triangular faces.
Since this set (as a subset of $\R^3$) does not contain, for example, $(0.8,0.5,0.2)$, we have that $\Theta_3$ is strictly smaller than the set of all correlation matrices with non-negative entries.
}
\end{example}
  
Theorem \ref{thm:ch:mat} also implies that all invariant correlation matrices are realized by the model \eqref{eq:newmodel}.
Nevertheless, not all exchangeable random vectors with an invariant correlation matrix can be represented by \eqref{eq:newmodel}.
For instance, any $d$-dimensional random vector with pairwise independent components has an invariant correlation matrix $I_d$ but is not necessarily modelled by \eqref{eq:newmodel}.
Generally, it is not possible to completely identify a multivariate model from bivariate properties.  A class of copulas whose bivariate marginals are Fr\'echet copulas is studied by \cite{YQW09}.
  
We next consider the {membership testing problem} of $\Theta_d$, which aims at determining whether a given correlation matrix $R=(r_{ij})_{d\times d}$ belongs to $\Theta_d$; that is, $R$ is realized as an invariant correlation matrix of some $d$-dimensional random vector.
Note that the number of vertices of the clique partition polytope equals the number of different partitions of $[d]$, and this number is known as the {Bell number} $N_d$, which can be computed explicitly
~\citep{GW90}.
The number $N_d$ grows rapidly in $d$. For example, $N_3=5$, $N_4 = 15$, $N_5 = 52$ and $N_{10} >10^5$.
By Theorem~\ref{thm:ch:mat}, we have $R=(r_{ij})_{d\times d}\in \Theta_d$ if and only if there exists $\alpha_1,\dots,\alpha_{N_d}\ge0$ such that $\sum_{\ell=1}^{N_d} \alpha_{\ell}=1$ and  $\sum_{\ell=1}^{N_d}\alpha_{\ell} \Sigma^{(\ell)}=R$, where $\mathcal S_{d,d}=\left\{\Sigma^{(\ell)}:\ell\in[N_d]\right\}$, $\Sigma^{(\ell)}=(\sigma_{ij}^{(\ell)})_{d\times d}$, is a collection of all vertices of the clique correlation polytope.
These linear constraints on $\boldsymbol{\alpha}$ can be summarized into $V_d \boldsymbol{\alpha}=\tilde{\mathbf{r}}$, 
where $\tilde{\mathbf{r}}\in [0,1]^{\tilde d}$ and $V_d\in \{0,1\}^{\tilde d\times N_d}$ are defined in~\eqref{eq:r:vec} below with $\tilde d =1+d(d-1)/2$.
By introducing an auxiliary variable $\mathbf{z}\in \R^{\tilde d}$, 
a given correlation matrix $R=(r_{ij})_{d\times d} \in [0,1]^{d\times d}$ is in $\Theta_d$ if and only if the following linear program attains zero:
  \begin{align}
    \label{eq:LP} \min_{\mathbf{z}\in \R^{\tilde d},\boldsymbol{\alpha}\in \R^{N_d}} 
    \mathbf{1}_{\tilde d}^\top \mathbf{z}
    \quad\text{subject to}\ \begin{cases}
      V_d\boldsymbol{\alpha} + \mathbf{z} = \tilde{\mathbf{r}},\\
\boldsymbol{\alpha}\ge \mathbf{0}_{N_d}\text{ and }\mathbf{z} \geq \mathbf{0}_{\tilde d},
    \end{cases}
  \end{align}
 where \begin{align}\label{eq:r:vec}     
  \tilde{\mathbf{r}} =(r_{12},r_{13},r_{23},\dots,r_{d-1\,d},1),\qquad
    V_d =
    \begin{pmatrix}
      \sigma_{12}^{(1)}  & \sigma_{12}^{(2)} &\cdots & \sigma_{12}^{(N_d)} \\
      \sigma_{13}^{(1)}  & \sigma_{13}^{(2)} &\cdots & \sigma_{13}^{(N_d)} \\
      \sigma_{23}^{(1)}  & \sigma_{23}^{(2)} &\cdots & \sigma_{23}^{(N_d)} \\
      \vdots                  & \vdots                     &\vdots &  \vdots \\
      \sigma_{(d-1)\,d}^{(1)} &  \sigma_{(d-1)\,d}^{(2)}  &\cdots & \sigma_{(d-1)\,d}^{(N_d)} \\
      1                  & 1                     &\cdots &    1\\
    \end{pmatrix}.
  \end{align}
Note that any correlation matrix with at least one negative entry is immediately excluded from $\Theta_d$.
For $R=(r_{ij})_{d\times d} \in [0,1]^{d\times d}$, 
the set of constraints in~\eqref{eq:LP} is always nonempty since it contains the pair
$(\boldsymbol{\alpha},\mathbf{z})=(\mathbf{0}_{N_d},\tilde{\mathbf{r}})$.  
The above linear program is solved, for example, with the \textsf{R}\ package \textsf{lpSolve}
although it can be computationally demanding for large $d$. 

\begin{remark}
\label{prop:simulation}
Suppose that the linear program~\eqref{eq:LP} attains $0$ at $\boldsymbol{\alpha}=\boldsymbol{\alpha}^\ast$ for a given correlation matrix $R=(r_{ij})_{d\times d}$.
Then one can simulate a $d$-dimensional random vector with invariant correlation $R$ as a mixture of the models of the form~\eqref{eq:model:cpp} with $\boldsymbol{\alpha}^\ast$ being the vector of mixture weights to $\{\Sigma^{(\ell)}:\ell\in[N_d]\}$.
\end{remark}


\subsection{Relationship to positive regression dependence}\label{sec:PRD}

 {Positive regression dependence}~\citep{L66} is a concept of dependence known to be useful to control the {false discovery rate} (FDR) through the procedure of \cite{BH95} in the context of multiple testing (\citealp{BY01}).
In this section, we explore the relationship between this dependence property and invariant correlation.

We start with the definition.
A set $A \subseteq \mathbb{R}^d$ is said to be {increasing} if $\mathbf{x} \in A$ implies $\mathbf{y}\in A$ for all $\mathbf{y}\geq \mathbf{x}$. 

\begin{definition}\label{def:prds}
A $d$-dimensional random vector $\mathbf{X}=(X_1,\dots,X_d)$ is said to have {positive regression dependence on the subset  $\mathcal N \subseteq [d]$ (PRDS)}  if for any index $i\in \mathcal N$ and increasing set $A\subseteq \mathbb{R}^d$, the function $x \mapsto \p(\mathbf{X}\in A \mid X_i = x)$ is increasing. 
The case of $\mathcal N=[d]$ is simply called {positive regression dependence (PRD)}. 
\end{definition}

Let us first consider $d=2$.
A direct consequence from Lemma~\ref{lem:quasi_Fre} in Appendix~\ref{app:2} is that every exchangeable random vector $(X_1,X_2)$ with invariant correlation $r\ge 0$ is PRD since~\eqref{eq:identical:identity:set} implies that, for every increasing $A$ and $x \in\R$, we have
$
\p\left((X_1,X_2)\in A \mid X_1 = x\right)=r 
\id_{\{(x,x)\in A\}}  + (1-r)\p\left((x,X_2)\in A\right)
$, which is increasing in $x$.
Note, however, that without exchangeability PRD is not implied by invariant correlation; see Example~\ref{ex:not:prd:2} in Appendix~\ref{app:auxiliary:prd}. 
 
For $d\ge 3$, there also exists an exchangeable model which has an invariant correlation matrix with non-negative entries but does not have PRD; this is reported in Example~\ref{ex:not:prds} in Appendix~\ref{app:auxiliary:prd}. 
Therefore, a dependence structure admitting an invariant correlation matrix does not imply PRD even if exchangeability is additionally assumed.
This is not surprising as invariant correlation is essentially a pairwise property, and 
pairwise dependence does not determine  overall dependence.
Nevertheless, the next proposition shows that the model~\eqref{eq:newmodel} has PRD, and hence also PRDS for any subset $\mathcal N\subseteq [d]$.  
Together with Theorem~\ref{thm:ch:mat}, this result indicates that every invariant correlation matrix admits a model with PRD. 

  \begin{proposition}\label{prop:X:prd}
  The random vector $\mathbf X=\Gamma \mathbf{U}$ in~\eqref{eq:newmodel} has PRD. 
  \end{proposition}

  Since both the properties of PRD  and invariance correlation are preserved under increasing transforms, we immediately obtain that the model 
\begin{align}\label{eq:discrete}
        \mathbf X = \Gamma \mathbf V, \quad\mathbf V=(V_1,\dots,V_d)^\top,\quad V_i=g(U_i) \quad\mbox{for~} i\in [d], 
  \end{align}
 has PRD and an invariant correlation matrix for any increasing and admissible $g$ with $g(0)=0$. 
 For example, by choosing $g(x) =\lceil nx \rceil /n$, we obtain the discrete version of \eqref{eq:newmodel}
where $V_1,\dots,V_d$ are independent and uniformly distributed on $[n]/n$.

\section{Invariant correlation under increasing transforms}\label{sec:ic:inc}

In many applications, only increasing transforms of the random variables are relevant, as non-increasing transforms do not preserve the copula among random variables. In view of this, we study a variant of invariant correlation where we require \eqref{eq:inv:corr} to hold only for increasing transforms, instead of all transforms. By definition, this requirement is weaker 
than invariant correlation defined in Section \ref{sec:2}. The interesting question is 
then whether this formulation allows for more models than those characterized in Section \ref{sec:3}. 


\begin{definition}
Let $r \in [-1,1]$.
A bivariate random vector $(X,Y)$ is said to have an {invariant correlation} $r$ {under increasing transforms}
if~\eqref{eq:inv:corr} holds for all admissible increasing functions $g$.
The set of all such random vectors is denoted by $\mathrm{IC}^{\uparrow}_r$.
\end{definition}

Denote by $\mathrm{IC}^{\uparrow}=\bigcup_{r\in [-1,1]}\mathrm{IC}^{\uparrow}_r$.
Note that $\mathrm{IC}_r \subseteq \mathrm{IC}^{\uparrow}_r$ for every $r\in [-1,1]$, and $\mathrm{IC} \subseteq \mathrm{IC}^{\uparrow}$ by their definitions.

 We summarize in Table \ref{table:increase} the results on the relationship between  $\mathrm{IC}$ and $\mathrm{IC}^{\uparrow}$.
 \begin{table}
\caption{
Relationship between $\mathrm{IC}$ and $\mathrm{IC}^{\uparrow}$, where $\mathrm{IC} \subseteq \mathrm{IC}^{\uparrow}$ always holds. The abbreviations mean  sets corresponding to those in Table \ref{table:IC}. 
Note that QI is equivalent to IN if at least one of the marginals is bi-atomic (Proposition~\ref{pro:ic0:one_bi}).
We slightly abuse the notation so that equality is understood between corresponding subsets.}\label{table:increase}
\def\arraystretch{1.3}
\vskip-0.3cm\hrule
\smallskip
\centering\small
\begin{tabular}{cccc}
 \multicolumn{2}{c}{Marginals $F, G$}& $r=0$ &$r\neq 0$\\
 \hline
 \multicolumn{2}{l}{$F=G$} &  \multirow{3}{*}{\shortstack{$\mathrm{IC}_0=\mathrm{IC}_0^{\uparrow}=\mathrm{QI}$ (Theorem \ref{thm:icincrease0})}} & $\mathrm{IC}_r=\mathrm{IC}^{\uparrow}_r=\mathrm{QF}$ (Theorem \ref{thm:icincrease0})  \\ 
\cline{1-2} 
 \cline{4-4}
 \multirow{2}{*}{$F\neq G$}   & both bi-atomic &
 & $\mathrm{IC}_r\ne \mathrm{IC}^{\uparrow}_r$ (Example \ref{ex:bernoulli}) \\
 \cline{2-2} 
\cline{4-4}
 &  not both bi-atomic & 
 & $\mathrm{IC}_r= \mathrm{IC}_r^{\uparrow}=\varnothing$  (Theorem \ref{thm:icincrease0})\\
\end{tabular}
\hrule
\end{table}

Our main message is that invariant correlation confined to increasing transforms does not accommodate more models except for the case when both random variables are bi-atomic.
Therefore, in most cases, we can safely treat the two formulations as equivalent.

 The results  of Proposition~\ref{lem:h} and Corollary~\ref{cor:copula} can be easily extended to $\mathrm{IC}^{\uparrow}_r$.
We collect these useful properties  in the following corollary.
The proofs are analogous to those of Proposition~\ref{lem:h} and Corollary~\ref{cor:copula}, and thus omitted.

\begin{corollary}\label{cor:ic:inc:basic:properties}
The following statements hold.
\begin{enumerate}[(i)]
\item\label{item:h:inc}
Let $(X,Y)\in \mathrm{IC}^{\uparrow}_r$ for some $r \in [-1,1]$, and let $h$ be an admissible increasing function for $(X,Y)$.
Then, $(h(X),h(Y))\in \mathrm{IC}^{\uparrow}_r$.
\item\label{item:cop:inc} Suppose that $(X,Y)\sim H$ has identical continuous and strictly increasing marginal distributions, a correlation coefficient $r \in [-1,1]$, and a copula $C$. Then $H\sim \mathrm{IC}^{\uparrow}_r$ if and only if $C\sim \mathrm{IC}^{\uparrow}_r$.
\end{enumerate}
\end{corollary}


We first show in the following example that the property of invariant correlation under increasing transforms is not always equivalent to that under all transforms.

\begin{example}\label{ex:bernoulli}
Let $X$ and $Y$ be non-identically distributed bi-atomic random variables.
In contrast to the result in Proposition~\ref{pro:bi-atomic}, $(X,Y)$ always admits an invariant correlation under increasing transforms.
This is because the increasingness of $g$ implies that $g(X)$ and $g(Y)$ are increasing linear functions of $X$ and $Y$, respectively.
Hence, $(X,Y)\in \mathrm{IC}^{\uparrow}$ regardless of the dependence structure of $(X,Y)$. As a result, we have $\mathrm{IC}\subsetneq \mathrm{IC}^{\uparrow}$.
\end{example}

Except for the bi-atomic distributions, we show in the next theorem that $\mathrm{IC}_r$ and $\mathrm{IC}_r^{\uparrow}$ are equivalent. 

\begin{theorem}\label{thm:icincrease0}
Let $X,Y \in \mathcal L^2$ and $r=\corr(X,Y)$.
Assume that one of the following conditions holds:
(i) $r=0$; (ii) $X$ and $Y$ have identical distributions; (iii) $X$ and $Y$ have different distributions and $|\supp(Y)|>2$.
Then $(X,Y)\in \mathrm{IC}_r$   if and only if $(X,Y) \in \mathrm{IC}^{\uparrow}_r$.  In particular, if (iii) holds, then $r=0$.

\end{theorem}
Theorem \ref{thm:icincrease0} also justifies our focus on $\mathrm{IC}$ throughout the paper, although in many applications relevant transforms are confined to be increasing.
\begin{remark}
By Item~(\ref{item:cop:inc}) in Corollary~\ref{cor:ic:inc:basic:properties}, the result  in Corollary~\ref{cor:ic:cont:identical}  remains valid with $\mathrm{IC}_r$ in the statements replaced by $\mathrm{IC}^{\uparrow}_r$.
\end{remark}



Under the condition~(iii), Theorem \ref{thm:icincrease0} further implies that, for $r\ne 0$, 
 there is no model $(X,Y)$ in $ \mathrm{IC}_r^{\uparrow}$ or $ \mathrm{IC}_r$.



\section{Concluding remarks}\label{sec:conclusion} 

Our main results on a full characterization of models with invariant correlation can be briefly summarized below. Except for the very special case of bi-atomic distributions, invariant correlation is characterized by  quasi-independence (Theorem \ref{thm:ch:zero:ic}) and quasi-Fr\'echet models (Theorem~\ref{thm:main:sec:4}), with non-identical marginal distributions excluded unless the correlation is zero (Theorem \ref{thm:main:sec:3}). 
The same holds true when transforms are confined to be increasing (Theorem \ref{thm:icincrease0}). 
We also identify the set of all compatible  invariant correlation matrices (Theorem \ref{thm:ch:mat}).   

Several aspects and generalizations of invariant correlation require further research. For instance, it would be interesting to understand whether a higher dimensional correlation measure, instead of the matrix of bivariate correlations, can be used to naturally formulate an invariant correlation property, and whether it has interesting implications similar to our results in the bivariate case. 
Since PRD and PRDS are dependence concepts that are not determined by their bivariate marginals, a higher dimensional notion of invariant correlation may be more naturally connected to PRD and PRDS. 
Although a connection of invariant correlation to multiple testing and FDR control has been briefly discussed in Appendix~\ref{app:conformal:pvalues},  its applications and relevance for statistics are not yet clear. 
A further question concerns whether restricting the marginal transforms to a smaller but practically relevant class would characterize different and   potentially useful models. These questions require further  studies.

\section*{Acknowledgements}
We are grateful to the Editor and an Associate Editor for their valuable comments.
Takaaki Koike was supported by JSPS KAKENHI Grant Numbers JP21K13275 and JP24K00273. Ruodu Wang acknowledges financial support from the Natural Sciences and Engineering Research Council of Canada (RGPIN-2024-03728) and Canada Research Chairs (CRC-2022-00141).

\appendix

\section*{Appendices}

\section{Auxiliary results}\label{app:auxiliary}

\subsection{Invariant correlation matrices and other dependence matrices}\label{app:other:measures}

In this section, we summarize the connection of $\Theta_d$, the set of all possible $d \times d$ invariant correlation matrices of continuous random vectors with identical marginals, to
the collection of tail-dependence matrices and other dependence matrices.
Let  $(X_1,X_2)$ be a bivariate random vector with continuous marginal distributions $F_1$ and $F_2$.
The {(lower) tail-dependence coefficient} of $(X_1,X_2)$ is defined by 
$\lambda=\lim_{u\downarrow 0}\p\left(F_1(X_1)\le u,F_2(X_2)\le u\right)/u$
given that the limit exists.
For a function $g:\R\rightarrow \R$, the {$g$-transformed rank correlation} of $(X_1,X_2)$ is defined by
$
\kappa_g(X_1,X_2)=\corr(g(F_1(X_1)),g(F_2(X_2))),
$
provided that it is well defined. 
The {tail-dependence matrix} of a $d$-dimensional random vector $\mathbf{X}=(X_1,\dots,X_d)$ with continuous marginal distributions has the tail-dependence coefficient of $(X_i,X_j)$ in its $(i,j)$th entry for  $i,j\in[d]$. 
Analogously, the $\kappa_g$-matrix of $\mathbf{X}$ is a $d \times d$ matrix whose $(i,j)$th entry is $\kappa_g(X_i,X_j)$ for $i,j\in[d]$.
The set of all possible $d\times d$ tail-dependence matrices is called the {tail-dependence compatibility set} and is denoted by $\mathcal T_d$; see~\cite{EHW16,FSS17,KSSW18,ST20} for recent studies.
Similarly, the $d$-dimensional {compatibility set} of $\kappa_g$, denoted by $\mathcal K_d(g)$, is the collection of all possible $\kappa_g$-matrices in dimension $d$.


The next proposition shows that, for $\mathbf X$ with identical continuous marginals and an invariant correlation matrix $R$, the model $\mathbf X$ has the same tail-dependence matrix and  $\kappa_g$-matrix $R$.

\begin{proposition}\label{prop:tdm}
Let $\bX$ be a $d$-dimensional continuous random vector with identical marginals.
If $\bX\in \mathrm{IC}^d$, then its correlation matrix and its tail-dependence matrix coincide, as well as its $\kappa_g$-matrices for any $g$ admissible for the standard uniform distribution.
\end{proposition}

 Proposition~\ref{prop:tdm} implies that $\Theta_d\subset \mathcal T_d$ and that $\Theta_d\subset \mathcal K_d(g)$ for wide varieties of $g$.
For instance, $\Theta_d$ is a subset of the compatibility sets of Blomqvist's beta and Spearman's rho, studied by \cite{HK19} and~\cite{WWW19}, respectively. 
Since the compatibility set of Blomqvist's beta is equal to that of {Kendall's tau}~\citep{MNS22}
and is smaller than that of {Gini's gamma}~\citep{KH22}, our result implies that
$\Theta_d$ is contained in all of these compatibility sets. 
 The connection of the clique partition polytope to $\mathcal T_d$ is given in Proposition~22 of~\cite{FSS17}.
  This result, together with Theorem~\ref{thm:ch:mat} and Proposition~\ref{prop:tdm},
  leads to the following relationship between $\Theta_d$ and $\mathcal T_d$: 
      $\Theta_d=\mathcal T_d$ for $d\le 4$, and  $\Theta_d\subsetneq\mathcal T_d$ for $d\ge  5$.

\subsection{Conformal p-values}\label{app:conformal:pvalues}


Conformal p-value, studied by \cite{BCLRS23} in the context of outlier detection, is an interesting example of invariant correlation.
Let $S_1,\dots,S_{n}$ be scores of the null training sample (computed from some score function on the data), where  $n$ is a fixed   positive integer. 
Scores of the test sample are denoted by  $S_{n+i}$, $i\in [d]$, and the corresponding conformal p-value is given by
\begin{align}\label{eq:conform}
P_i=\frac{1}{n+1}+\frac{1}{n+1}\sum_{k=1}^n \id_{\{S_k\le S_{n+i}\}}.
\end{align}
Assume that $\{S_i:i\in [n+d]\}$ is a sequence of independent random variables and $S_k\sim F_0$, $k\in [n]$.
The above p-values are testing the sequence of null hypotheses H$_{0,i}:S_{n+i}\sim F_0$, $i\in [d]$. 

\begin{proposition}\label{prop:conformal}
    Let $\mathcal N\subseteq [d]$ be any set of null indices, that is, \rm{H}$_{0,i}$ is true for $i\in \mathcal N$.
        Then,
    \begin{align*}
        \p\left(P_i=\frac{j_i}{n+1},~i\in \mathcal N\right)=\frac{N_1!\cdots N_{n+1}!}{(n+m)(n+m-1)\cdots(n+1)},
        \quad
        j_i \in [n+1],
    \end{align*}
    where $m=|\mathcal N|$ and $N_j=|\{i\in \mathcal N: j_i=j\}|$, $j\in[n+1]$. 
    In particular, $P_i$ is uniformly distributed on $[n+1]/(n+1)$ for every $i\in \mathcal N$.
    Moreover,  $(P_i)_{i\in \mathcal N}$   has the invariant correlation matrix $(r_{ij})_{m\times m}$ such that $r_{ij}=1/(n+2)$ for every $i,j\in [m]$, $i\neq j$.
\end{proposition}


We can take $\mathcal N$ to be the set of all null indices in Proposition \ref{prop:conformal}. 
The fact that $(P_i)_{i\in \mathcal N}$ has an invariant correlation matrix has been shown in Lemma 2.1 of \cite{BCLRS23}, which also contains the joint distribution of $(P_i, P_j)$ for $i,j\in \mathcal N$. 
Proposition \ref{prop:conformal} further gives the joint distribution of  $(P_i)_{i\in \mathcal N}$. 
On the other hand, the full vector of p-values, $(P_i)_{i\in [d]}$  does not have an invariant correlation matrix in general. Note that the marginal distributions of $(P_i)_{i\in [d]}$ are generally different for null and non-null components, and there is some positive correlation between them by design in \eqref{eq:conform}. Hence, our results in Section \ref{sec:3} explain that invariant correlation for $(P_i)_{i\in [d]}$  is not possible.
Nevertheless, the vector of conformal p-values has PRDS on $\mathcal N$ as shown by Theorem~2.4 of \cite{BCLRS23}.

For $m=2$, the vector of null conformal p-values follows the model~\eqref{eq:discrete} with $k=n+2$, $g(x)=\lfloor(n+1)x\rfloor/(n+1)$ and $\mathbf{Z}_1,\mathbf{Z}_2$ being iid multinomials with the number of trials $1$ and uniform event probabilities on $[k]$.
It is left open whether the vector of null conformal p-values can be written of the form~\eqref{eq:discrete} for $m\ge 3$.



\subsection{Invariant correlaton and concepts of dependence}\label{app:auxiliary:prd}

In Section~\ref{sec:PRD}, we have explored the relationship between invariant correlation and positive regression dependence (PRD).
In this section, we show by examples that models with invariant correlation matrix do not always have PRD. 

We start from the bivariate case $d=2$.
Following~\cite{L66}, a random vector $(X,Y)$ is called {positive quadrant dependent (PQD)} (also called {positive orthant dependent}), if $\p(X\le x,Y\le y)\ge \p(X\le x)\p(Y\le y)$ for every $(x,y)\in \R^2$.
{Negative quadrant dependence (NQD)} is analogously defined by $\p(X\le x,Y\le y)\le \p(X\le x)\p(Y\le y)$ for every $(x,y)\in \R^2$.
It is shown in Lemma~4 of~\cite{L66} that PRD implies PQD.
In the next example, we construct a random vector that has zero invariant correlation but is neither PQD nor NQD.
This example particularly indicates that non-negative invariant correlation does not lead to PRD.

\begin{example}\label{ex:not:prd:2}
Let $(X,Y)$ be the model in Example~\ref{cor:tri0} with the identical marginal support $[3]$, $p_i=q_i=1/3$ for $i \in [3]$ and $\epsilon = 1/9$.
Then $\p(X\le 2,Y\le 1)=1/9$, $\p(X\le 1,Y\le 2)=1/3$ and $\p(X\le 1)\p(Y\le 2)=\p(X\le 2)\p(Y\le 1) = 2/9$. Hence $(X,Y)$ is neither PQD nor NQD.
\end{example}

We have seen in Section~\ref{sec:PRD} that, for the bivariate case, invariant correlation implies PRD under the additional assumption of exchangeability.
The next example shows that this is not the case for $d= 3$, that is, there exists an exchangeable model which admits a non-negative invariant correlation matrix but does not have PRD.

\begin{example}\label{ex:not:prds}
Let $d=3$ and consider a member of the Farlie-Gumbel-Morgenstein family of copulas:
 \begin{align*}
C_{\theta}(u_1,u_2,u_3)=u_1 u_2 u_3 + \theta u_1 u_2 u_3(1-u_1)(1-u_2)(1-u_3),\quad \theta \in [-1,1].
 \end{align*}
 Note that $C_\theta$ is an exchangeable copula for any $\theta \in [-1,1]$.
 Moreover, all bivariate marginals of $C_\theta$ are $\Pi$, and thus it has an invariant correlation matrix $I_3$.

  We will show that the model $\mathbf{U}\sim C_{\theta}$ does not satisfy the following weaker version of PRDS than that in Definition~\ref{def:prds}:
 \begin{align*}
s\mapsto \p(\mathbf{U}\in A\mid U_i\le s)\text{ is increasing in $s$ for any index $i\in \mathcal N$ and increasing set $A\subseteq \mathbb{R}^d$}.
 \end{align*}  
To this end, let $i=1$, $s \in [0,1]$, $\mathbf{t}=(t_1,t_2,t_3)\in[0,1]^3$, $t_1=0$, and  $A=(t_1,\infty)\times (t_2,\infty)\times (t_3,\infty)$.
Then $A$ is an increasing set and
\begin{align*}
\p(\mathbf{U} \in A\mid U_1 \le s) 
&=\p(\mathbf{U} > \mathbf{t}\mid U_1 \le s) \\
&= 
\frac{1}{s}
\p(U_1 \le s\,,t_2 < U_2,\, t_3 < U_3)\\
&= \frac{1}{s}\left\{
(1-t_2)(1-t_3) - 
\overline C(s,t_2,t_3)
\right\},
\end{align*}
where $\overline C(\mathbf{u})= \p(\mathbf{U}>\mathbf{u})$ is the joint survival function of $C$.
By calculation,
we have that
$
\p(\mathbf{U}\in A\mid U_1 \le s) 
= 1-t_2 - t_3 + C(s,t_2,t_3)/s.$
Therefore,
\begin{align*}
\frac{\partial}{\partial s}\p(\mathbf{U} \in A \mid U_1 \le s)
 = \frac{s\, \partial_1 C(s,t_2,t_3) - C(s,t_2,t_3)}{s^2}.
 \end{align*}
  When $t_2=t_3 = 0.5$ and $\theta = 1$, this derivative is a constant $-1/16$ by a simple calculation.
Therefore, $\p(\mathbf{U} \in A \mid U_1 \le s)$ is a decreasing function in $s\in [0,1]$.
\end{example}

\section{Proofs of all results}
\subsection{Proofs in Section \ref{sec:2}}\label{app:2}

\begin{proof}[\textbf{\upshape Proof of Proposition~\ref{lem:h}}]
Let $(X,Y)\in \mathrm{IC}_r$ and  $g$ be an admissible function for $(h(X),h(Y))$.
The function $g\circ h$ is admissible for $(X,Y)$.
Since $(X,Y)$ has an invariant correlation $r$, we have 
$
\corr(g\circ h(X),g\circ h(Y))=\corr(X,Y)=r.
$
Since an admissible function $g$ is arbitrary, we have that $(h(X),h(Y))\in \mathrm{IC}_r$.
Applying this observation to all pairs of components of $(X_1,\dots,X_d)$ establishes the desired result. 
\end{proof}

\begin{proof}[\textbf{\upshape Proof of Corollary \ref{cor:copula}}]
Suppose that $\mathbf X\in \mathrm{IC}^d$ with continuous and strictly increasing marginals $F$.
By taking $h=F$ in Proposition~\ref{lem:h}, we have that $(F(X_1), \dots, F(X_n))\in \mathrm{IC}^d$; hence $C \sim \mathrm{IC}^d$.
Suppose next that $C \sim \mathrm{IC}^d$. There exists a uniform random vector $\mathbf U =(U_1, \dots, U_n)\in \mathrm{IC}^d$. Let $\mathbf X=(F^{-1}(U_1), \dots, F^{-1}(U_n))$. 
By taking $h=F^{-1}$ in Proposition~\ref{lem:h}, we have that $\mathbf X \in \mathrm{IC}^d$.
\end{proof}

The next lemma justifies that the formulations of quasi-independence via \eqref{eq:quasi} and \eqref{eq:quasi:2}
are equivalent. 

\begin{lemma}\label{lem:quasi}
    For a random vector $(X,Y) \sim H$ with marginals $F$ and $G$, $(X,Y)$ is quasi-independent if and only if  \eqref{eq:quasi:2} holds for all $A,B\in \mathcal B(\R)$.
\end{lemma}

\begin{proof}[\textbf{\upshape Proof of Lemma \ref{lem:quasi}}]
It is clear that   \eqref{eq:quasi:2} implies quasi-independence. We will show the ``only if" part.  
Assume that $(X,Y)$ is quasi-independent. For $B \in \mathcal{B}(\R)$, define 
$\mathcal{D}_B(X,Y)=\{A\in\mathcal{B}(\R):\eqref{eq:quasi:2}~ \mbox{holds for } A,B\}.$
In what follows, we check that $\mathcal{D}_B(X,Y)$ is a $\lambda$-system for any $B\in\mathcal{B}(\R)$.
\begin{enumerate}
\item It is clear that $\R\in \mathcal{D}_B(X,Y)$.
\item Let $A_1, A_2\in \mathcal{D}_B(X,Y)$ such that $A_1\subseteq A_2$. For $A_2\setminus A_1$, we have 
\begin{align*}
&\p(X\in A_2\setminus A_1, Y\in B )+\p(X\in B, Y\in A_2\setminus A_1, )\\
&=\p(X\in A_2, Y\in B )-\p(X\in A_1, Y\in B )+\p(X\in B, Y\in A_2)-\p(X\in B, Y\in A_1 )\\
&=\p(X\in A_2)\p( Y\in B )-\p(X\in A_1)\p( Y\in B )+\p(X\in B)\p( Y\in A_2)-\p(X\in B)\p( Y\in A_1 )\\
&=\p(X\in A_2\setminus A_1)\p( Y\in B )+\p(X\in B)\p( Y\in A_2\setminus A_1).
\end{align*}
Hence, $A_2\setminus A_1\in \mathcal{D}_B(X,Y)$.
\item Let $A_1\subseteq A_2\subseteq A_3\subseteq \cdots$ be an increasing sequence of sets in $\mathcal{D}_B(X,Y)$. For $\bigcup_{i=1}^\infty A_i$, we have
\begin{align*}
&\p\left(A\in \bigcup_{i=1}^\infty A_i, Y\in B\right)+\p\left(A\in B, Y\in \bigcup_{i=1}^\infty A_i\right)\\
&=\lim_{n \to \infty} \p(A\in  A_n, Y\in B)+\lim_{n \to \infty }\p(A\in B, Y\in  A_n)\\
&=\lim_{n \to \infty} \p(A\in  A_n)\p( Y\in B)+\lim_{n \to \infty} \p(A\in B)\p( Y\in  A_n)\\
&=\p\left(A\in \bigcup_{i=1}^\infty A_i\right)\p( Y\in B)+\p(A\in B)\p\left( Y\in \bigcup_{i=1}^\infty A_i\right).
\end{align*}
Hence, $\bigcup_{i=1}^\infty A_i \in\mathcal{D}_B(X,Y)$.
\end{enumerate}

Let $B=(-\infty,y]$ for some $y\in \R$.
Then quasi-independence of $(X,Y)$ implies that $L:=\{(-\infty,x]: x\in \R\}\subseteq \mathcal{D}_B(X,Y)$.
It is straightforward to check that $L$ is a $\pi$-system. Therefore, Sierpi\'nski–Dynkin's $\pi$-$\lambda$ theorem yields $\sigma(L) \subseteq \mathcal{D}_{B}(X,Y)$, where $\sigma(L)$ is the smallest $\sigma$-algebra containing $L$. 
As $\sigma(L)=\mathcal{B}(\R)$, we have $\mathcal{D}_{B}(X,Y)=\mathcal{B}(\R)$ for any $B=(-\infty, y]$ with $y\in \R$.

Next, let $\mathcal{D}(X,Y)=\{B\in \mathcal{B}(\R):\mathcal{D}_B(X,Y)=\mathcal{B}(\R)\}.$
It is clear that $\mathcal{D}(X,Y)$ is a $\lambda$-system and $L\subseteq\mathcal{D}(X,Y)$. Hence, we have $\mathcal{D}(X,Y)=\mathcal{B}(\R)$ by the same argument as above. Therefore, we obtain \eqref{eq:quasi:2} for all $A,B \in \mathcal{B}(\R)$.
\end{proof}

The next lemma justifies the formulations of the quasi-Fr\'echet model via \eqref{eq:identical:identity} and \eqref{eq:identical:identity:set}.
We omit the proof since it is analogous to that of Lemma~\ref{lem:quasi}.

\begin{lemma}\label{lem:quasi_Fre}
    For a random vector $(X,Y) \sim H$ with the marginals $F$, $(X,Y)$ is quasi-$r$-Fr\'echet if and only if  \eqref{eq:identical:identity:set} holds.
\end{lemma}

\subsection{Proofs in Section \ref{sec:ic0}}
\label{app:31}
\begin{proof}[\textbf{\upshape Proof of Theorem \ref{thm:ch:zero:ic}}]
We first show the ``only if" part. As  $(X,Y)\in\mathrm{IC}_0$, $\cov(g(X), g(Y))=0$ for any admissible $g$.
For any $A \in {\mathcal B}(\R)$, taking $g(x)=\id_A(x)$, we have
$\cov(g(X),g(Y))=\cov(\id_A(X),\id_A(Y))=0.$
For any $A, B\in {\mathcal B}(\R)$, let $g(x)=\id_{A}(x)+\id_{B}(x)$, which leads to 
\begin{align*}
\cov(g(X),g(Y))&=\cov\(\id_{A}(X)+\id_{B}(X),\id_{A}(Y)+\id_{B}(Y)\)\\
&=\cov\(\id_{A}(X),\id_{B}(Y)\)+\cov\(\id_{B}(X),\id_{A}(Y)\)\\
&=\p\(X \in A, Y\in B\)+\p\(X \in B, Y\in A\)\\
&~~~-\p\(X\in A\)\p\(Y\in B\)-\p\(X\in B\)\p\(Y\in A\).
\end{align*}
As $\cov(g(X),g(Y))=0$, we have that $(X,Y)$ satisfies \eqref{eq:quasi:2} for all $A,B \in \mathcal{B}(\R)$. Hence, $(X,Y)$ is quasi-independent by Lemma \ref{lem:quasi}.

Next, we show that quasi-independence implies  $\cov(g(X),g(Y))=0$ for any admissible $g$.
As quasi-independence implies \eqref{eq:quasi:2}, by taking $A=B$, we have 
 $
  \p(X \in A, Y\in A)=\p(X\in A)\p(Y \in A)
 $
 for any $A\in {\mathcal B}(\R)$.
Hence, $\cov(g(X), g(Y))=0$ for any indicator functions $g=\id_{A}$. 

First, let $g(x)=\sum_{i=1}^n c_i \id_{A_i}(x)$ be an admissible simple function   where $A_1,\dots, A_n \subseteq \R$ are disjoint measurable sets, and  $c_1, \dots c_n\in \R$  are real numbers. For the simple function $g$, we have
$$
\begin{aligned}
\cov(g(X),g(Y))
&=\sum_{i=1}^n\sum_{j=1}^n c_ic_j\cov\( \id_{A_i}(X), \id_{A_j}(Y)\)\\
&=\sum_{1\le i<j\le n} \(c_ic_j\cov\( \id_{A_i}(X), \id_{A_j}(Y)\)+c_jc_i\cov\( \id_{A_j}(X), \id_{A_i}(Y)\)\)\\
&=\sum_{1\le i<j\le n} c_ic_j\bigg(\p\( X\in A_i,Y \in A_j\)-\p\( X\in A_i\)\p\(Y \in A_j\)\\
&~~~+\p\( X \in A_j, Y\in A_i\)-\p\(X\in A_j\)\p \(Y\in A_i\)\bigg)=0.
\end{aligned}$$
Therefore,  we have $\cov(g(X),g(Y))=0$ for all admissible simple functions $g$.

Second, let $g$ be an admissible non-negative  function. Admissibility of $g$ implies that
$\E[g(X)]<\infty$, $\E[g(Y)]<\infty$ and $\E[g(X)g(Y)]<\infty$.
For a non-negative admissible function $g$, there exists a sequence of non-negative simple functions $\{g_n\}_{n\ge 1}$  such that $g_n \uparrow g$ pointwise.
Thus, we can get
$$\begin{aligned}
\cov\(g(X),g(Y)\)&=\cov\(\lim_{n \to \infty} g_n(X),\lim_{n \to \infty} g_n(Y)\)\\
&=\E\left[\lim_{n \to \infty} g_n(X)\lim_{n \to \infty} g_n(Y)\right]-\E\left[\lim_{n \to \infty} g_n(X)\right]\E\left[\lim_{n \to \infty} g_n(Y)\right]\\
&=\E\left[\lim_{n \to \infty} g_n(X) g_n(Y)\right]-\E\left[\lim_{n \to \infty} g_n(X)\right]\E\left[\lim_{n \to \infty} g_n(Y)\right]\\
&=\lim_{n \to \infty}\E\left[ g_n(X) g_n(Y)\right]-\lim_{n \to \infty}\E\left[ g_n(X)\right]\E\left[ g_n(Y)\right]\\
&=\lim_{n \to \infty}\left\{\E\left[ g_n(X) g_n(Y)\right]-\E\left[ g_n(X)\right]\E\left[ g_n(Y)\right]\right\}\\
&=\lim_{n \to \infty}\cov(g_n(X), g_n(Y))=0.
\end{aligned}$$
Therefore, we have $\cov(g(X),g(Y))=0$ for all  admissible non-negative  function $g$.

Finally, let  $g$ be an admissible function. Define $g_+=\max(g,0)$ and $g_-=-\min(g,0)$. It is clear that $g_+$ and $g_-$ are non-negative admissible functions and $g=g_+-g_-$.
Let $G=g_++g_-$. We have that $G$ is also a non-negative admissible function. Hence,  we get $\cov(g_+(X),g_+(Y))=0$,  $\cov(g_-(X),g_-(Y))=0$ and $\cov(G(X),G(Y))=0$, which imply that
$
\cov(g_+(X),g_-(Y))+\cov(g_-(X),g_+(Y))=0.
$
As a result, we have
\begin{align*}
\cov\(g(X),g(Y)\)=\cov(g_+(X)-g_-(X),g_+(Y)-g_-(Y))=0.
\end{align*}
Therefore, we can conclude $(X,Y) \in \mathrm{IC}_0$.
\end{proof}

\begin{proof}[\textbf{\upshape Proof of Proposition \ref{pro:ic0:independent}}]
  For all $A,B \in \mathcal B(\R)$, we have $$\p(X_{\pi_1} \in A, X_{\pi_2}\in B)=\frac{1}{2}\p(X_1 \in A, X_2\in B)+\frac{1}{2}\p(X_1 \in B, X_2 \in A).$$
    Since $X_1$ and $X_2$ have the same distribution, independence of  $(X_{\pi_1}, X_{\pi_2})$ is equivalent to
  \begin{align*}
      \p(X_{\pi_1} \in A, X_{\pi_2}\in B)=\p(X_{\pi_1} \in A)\p( X_{\pi_2}\in B)
      =\p(X_1 \in A)\p(X_2 \in B)=\p(X_1 \in B)\p(X_2 \in A).
  \end{align*}
  Hence, independence of $(X_{\pi_1}, X_{\pi_2})$ is equivalent to 
  \begin{align*}\p(X_1 \in A, X_2\in B)+\p(X_1 \in B, X_2 \in A)=\p(X_1 \in A)\p(X_2 \in B)+\p(X_1 \in B)\p(X_2 \in A),
  \end{align*}
  which is also equivalent to $(X_1,X_2)\in \mathrm{IC}_0$ by Theorem \ref{thm:ch:zero:ic}.
\end{proof}

\begin{proof}[\textbf{\upshape Proof of Proposition \ref{cor:n:atomic}}]
  We can get  ``(i) $\Leftrightarrow$ (ii)" from Theorem \ref{thm:ch:zero:ic}. We only show ``(ii) $\Leftrightarrow$ (iii)".

 Suppose that $P+P^\top=\mathbf p\mathbf q^\top +\mathbf q\mathbf p^\top$.
 Let $S=P^\top-\mathbf q\mathbf p^\top$. We first verify that $S$ satisfies (a), (b) and (c).
 \begin{enumerate}[1.]
 \item As  $s_{ij}+p_iq_j=p_{ij}\ge 0$, we have that $S$ satisfies (a).
 \item The $j$th row sum of $S$ is $\sum_{i=1}^n p_{ij}-\sum_{i=1}^nq_jp_i=q_j-q_j=0$. The $i$th column sum of $S$ is $\sum_{j=1}^n p_{ij}-\sum_{j=1}^n p_i q_j=p_i-p_i=0$. Hence, $S$ satisfies (b).
 \item   We have $S+S^\top= P^{\top}+P-\mathbf q\mathbf p^{\top}-\mathbf p\mathbf q^{\top}=O$ where $O$ is an $n \times n$ matrix with all elements equal to 0. 
Hence, $S$ satisfies (c).   
 \end{enumerate}

Next, let $P=\mathbf p\mathbf q^\top+S$ with $S$ satisfies (a), (b) and (c). We verify that $P$ is a valid probability matrix and satisfies (ii).
First, it is clear that $p_{ij}\ge0$ as $p_{ij}=p_iq_j+s_{ij}\ge 0$.
Second, the marginals is $\mathbf q$ and $\mathbf p$ as  $\sum_{j=1}^n p_{ij}=\sum_{j=1}^n (p_iq_j+s_{ij})=p_i$ and $\sum_{i=1}^n p_{ij}=\sum_{i=1}^n (p_iq_j+s_{ij})=q_j$. 
Finally, since $S+S^\top=O$, we have $P+P^\top=\mathbf p\mathbf q^\top+S+ \mathbf q \mathbf p^\top+S^\top= \mathbf p\mathbf q^\top +\mathbf q\mathbf p^\top.$ 
Therefore, we can conclude ``(ii) $\Leftrightarrow$ (iii)". 
\end{proof}


\begin{proof}[\textbf{\upshape Proof of Proposition \ref{pro:ic0:one_bi}}]

It is clear that independent $(X,Y)$ satisfies invariant correlation with $\corr(X,Y)=0$. By Theorem \ref{thm:ch:zero:ic}, we only need to show that quasi-independence implies independence under the condition of Proposition \ref{pro:ic0:one_bi}. Without loss of generality, assume that the support of $X$ is $\{1,2\}$. To show that $X$ and $Y$ are independent, it is enough to show that $\p(X=1,Y\in A)=\p(X=1)\p(Y \in A)$ for all $A \in \R$.
We show this statement in each of the following 4 cases.

{Case 1}: Assume $\{1,2\} \subseteq A$. We have $\p(X \in A)=1$ and $\p(X \in A,Y=1)=\p(Y=1)$. Hence, $\p(X\in A, Y=1)=\p(X\in A)\p(Y =1)$. In this case, quasi-independence means
      $\p(X =1, Y\in A)+\p(X \in A, Y=1)=\p(X=1)\p(Y \in A)+\p(X\in A)\p(Y=1).$ Thus, we can  get $\p(X=1,Y\in A)=\p(X=1)\p(Y \in A)$.

{Case 2}: 
Assume $1 \notin A$ and $2 \notin A$. We have  $\p(X \in A)=0$ and $\p(X\in A, Y=1)=0$. Hence,  $\p(X\in A,Y=1)=\p(X \in A)\p(Y=1)$. Therefore, we have $\p(X=1,Y\in A)=\p(X=1)\p(Y \in A)$ using the same argument as in Case 1.

{Case 3}: 
Assume $1 \in A$, $2 \notin A$.
First, using Case 2, we have $\p(X=1,Y\in A\setminus \{1\})=\p(X=1)\p(Y\in A\setminus \{1\}).$  Furthermore, quasi-independence implies $\p(X=1,Y=1)=\p(X=1)\p(Y=1)$. Therefore, 
$$
    \p(X=1, Y\in A)=\p(X=1, Y=1)+\p(X=1,Y\in A\setminus \{1\})=\p(X=1)\p(Y\in A).$$

{Case 4}: 
 Assume $1 \notin A$, $2 \in A$. By the same argument as in Case 3, we have $\p(X=2, Y \in A)=\p(X=2)\p(Y \in A)$. Hence,
\begin{align*}
\p(X=1,Y\in A)&=\p(Y \in A)-\p(X=2,Y\in A)\\
&=\p(Y \in A)-\p(X=2)\p(Y\in A)\\
&=\p(X=1)\p(Y \in A).\end{align*}

In conclusion, we have $\p(X=1,Y\in A)=\p(X=1)\p(Y \in A)$ for all $A \subseteq \R$. As $\p(X=1)+\p(X=2)=1$, we can also obtain $\p(X=2,Y\in A)=\p(X=2)\p(Y \in A)$ for all $A \subseteq \R$. Hence, $X$ and $Y$ are independent.
\end{proof}

\subsection{Proofs in Section \ref{sec:non:identical}}\label{app:3.2}
The proof of Proposition \ref{pro:bi-atomic} is built on the following lemma.
\begin{lemma}\label{lem:bernoulli}
If $X$ and $Y$ are  bi-atomic random variables with the same support, then $(X,Y)\in \mathrm{IC}$.
\end{lemma}

\begin{proof}[\textbf{\upshape Proof of Lemma \ref{lem:bernoulli}}]
As the support of $X$ and $Y$ has only two points,  $g(X)$ and $g(Y)$ are linear transformations of $X$ and $Y$.
Therefore, we have
 $
  \corr(g(X),g(Y))= \corr(aX+b, aY+b)=\corr(X,Y)
 $
 for any admissible $g$. 
 Hence, $(X,Y)\in \mathrm{IC}$  regardless of the dependence structure of $(X,Y)$.
\end{proof}

\begin{proof}[\textbf{\upshape Proof of Proposition \ref{pro:bi-atomic}}]
The ``if'' part  of  (ii) is clear from Example \ref{ex:independent} and Lemma~\ref{lem:bernoulli}. We  show the ``only if" part below.

Let $\{x_1,x_2\}$ and $\{y_1,y_2\}$ be the support of $X$ and $Y$, respectively, with $x_1 < x_2$ and $y_1<y_2$. Assume $X$ and $Y$ have different supports. We can easily find a function $g$ such that $g(x_1)<g(x_2)$ and $g(y_1)>g(y_2)$.
Then $\corr(g(X), g(Y))=-\corr(X,Y)$ since $g(X)$ is an increasing linear transformation of $X$ and $g(Y)$ is  a decreasing linear transformation of $Y$.
Therefore, $(X,Y) \in \mathrm{IC}$ implies $\corr(X,Y)=0$. For $r\neq 0$, $(X,Y) \in \mathrm{IC}_r$ implies that $X$ and $Y$ have the same support.
If $(X,Y) \in \mathrm{IC}_0$, by Proposition \ref{pro:ic0:one_bi}, $X$ and $Y$ are independent.
 \end{proof}




Before showing Theorem \ref{thm:main:sec:3} for the general case, we first analyze the case of the tri-atomic distribution in the following lemma.

  \begin{lemma}\label{pro:tri}
 Suppose  $\supp(X) \subseteq \supp(Y)=\{x_1, x_2, x_3\}$.  
 If $(X,Y) \in \mathrm{IC}$, then   $F_X=F_Y$ or   $\corr(X,Y)=0$.
 \end{lemma}

 \begin{proof}[\textbf{\upshape Proof of Lemma \ref{pro:tri}}]
Since a linear transform of $g$ does not change $\corr(g(X),g(Y))$, it suffices to consider the case $g(x_1)=a$, $g(x_2)=0$, and $g(x_3)=1$.
 Let $p_{i}=\p(X=x_i)$, $q_j=\p(Y=x_j)$,  $p_{ij}=\p(X=x_i,Y=x_j)  $ and $s_{ij}=p_{ij}-p_iq_j$ for $i,j\in[3]$. As $\supp(X)\subseteq\supp(Y)=\{x_1, x_2, x_3\}$, we can assume that $0<p_i<1$ for $i=1,3$ and $0<q_i<1$ for $j\in[3]$.
 Write the matrices  $P=(p_{ij})_{3\times 3}$  and $S= (s_{ij})_{3\times 3}$,
 and they are connected by
$
P= \mathbf p \mathbf q^\top   +  S,
$
where $\mathbf p=(p_1,p_2,p_3)^\top$ and $\mathbf q=(q_1,q_2,q_3)^\top$.
Note that each of the row sums and column sums of $S$ is $0$, that is,
$\sum_{k=1}^3 s_{kj} = 0 = \sum_{k=1}^3 s_{ik},~~i,j\in[3]$.

Define the function $f:\R\to \R$ as
\begin{align*}
f(a):&=\corr(g(X),g(Y)) = \frac{a^2 p_{11}+ a (p_{13} + p_{31}) +  p_{33} - (ap_1 + p_3) (aq_1 + q_3) }{ \sqrt{a^2p_1 +   p_3 -(ap_1 + p_3)^2  } \sqrt{ a^2q_1 +  q_3 -(aq_1 + q_3)^2   } }
\\& = \frac{a^2  s_{11} + a ( s_{13}+ s_{31} ) +  s_{33}}{ \sqrt{a^2(p_1 -p_1^2) -  2a p_1p_3  + (p_3 -p_3^2)   }
\sqrt{a^2(q_1 -q_1^2) -  2a q_1q_3  + (q_3 -q_3^2)   }  }.
\end{align*}
If $(X,Y) \in \mathrm{IC}_r$, the equation $f(a)=r$ holds for all $a\in \R$. 
Hence, by matching the coefficients in front of $a^k$, $k\in [4]$, in the equation $f(a)=r$, we obtain the follow equations:
\begin{subequations}
\begin{align}
s_{11}^2 = r^2   (p_1 -p_1^2)(q_1-q_1^2),~~~~~~~
s_{11}(s_{13}+s_{31}) =  -  r^2  \left( q_1q_3   (p_1 -p_1^2) +  p_1p_3 (q_1 -q_1^2)\right),\label{6a}\\
2 s_{11}s_{33}  +(s_{13}+s_{31})^2  =  r^2  \left(  (p_1 -p_1^2)  (q_3 -q_3^2) + (p_3 -p_3^2) (q_1 -q_1^2)  +4 p_1q_1p_3q_3\right),\label{6c}\\
s_{33}(s_{13}+s_{31}) = -  r^2  \left(   p_1p_3(q_3 -q_3^2)   +   q_1q_3  (p_3 -p_3^2)    \right),~~~~~~~
s_{33}^2 =  r^2(p_3 -p_3^2)(q_3-q_3^2).\label{6d}
\end{align}
\end{subequations}
Next, we will show that if $r\neq 0$, we have $\mathbf p=\mathbf q$.





As $p_1, p_3, q_1 q_3$ take values in $(0,1)$ and $r \ne 0$, we have that $s_{11}$, $s_{33}$ and $ s_{13}+s_{31}$ are non-zero.
Hence, \eqref{6a} and  \eqref{6d}  yield
$$\frac{s^2_{11}}{s^2_{33}}=\frac{ \left(  q_1q_3   (p_1 -p_1^2) +   p_1p_3(q_1 -q_1^2) \right)^2 }{ \left( p_1p_3(q_3 -q_3^2)   +   q_1q_3  (p_3 -p_3^2)  \right)^2 }=\frac{(p_1 -p_1^2)(q_1-q_1^2)}{(p_3 -p_3^2)(q_3-q_3^2)}.$$
Simplifying the equation, we get
$$\left\{q_1^2q_3^2 (p_1 -p_1^2)(p_3 -p_3^2)-p_1^2p_3^2(q_1-q_1^2)(q_3 -q_3^2)\right\}\left\{(p_1 -p_1^2)(q_3-q_3^2)-(p_3 -p_3^2)(q_1-q_1^2)\right\}=0.$$
%
This equation is further simplified to
\begin{equation}\label{eq:tri}
\left(q_1q_3p_2-p_1p_3q_2\right)\left(p_2q_3-q_2p_3\right)=0,
\end{equation}
by using $p_1+p_2+p_3=1$ and $q_1+q_2+q_3=1$.
We can observe that if $p_2=0$, then \eqref{eq:tri} cannot hold. Hence, we consider the case when $0<p_2<1$ below.

In this case, we can switch the roles of indices $1,2,3$ as $p_i$ and $q_i$ take value in $(0,1)$ for all $i\in[3]$.
Hence,  we have $f(a)=r$ for all $a\in\R$ if and only if 
\begin{align}
(q_1q_3p_2-p_1p_3q_2) \left(p_2q_3-q_2p_3\right)=0,\label{pq1}\\
(q_2q_1p_3-p_2p_1q_3)\left(p_3q_1-q_3p_1\right)=0,\label{pq2}\\
(q_3q_2p_1-p_3p_2q_1 )\left(p_1q_2-q_1p_2\right)=0.\label{pq3}
\end{align}
Now we consider all possible cases when \eqref{pq1}, \eqref{pq2} and \eqref{pq3} hold simultaneously.

First, assume that all the first terms  in \eqref{pq1}, \eqref{pq2} and \eqref{pq3} equal 0;
that is $q_1q_3p_2=p_1p_3q_2$, $q_2q_1p_3=p_2p_1q_3$ and $q_3q_2p_1=p_3p_2q_1$. Thus, we have $q_1q_2q_3=p_1p_2p_3$. As a result, we get $p^2_1=q^2_1$, $p_2^2=q_2^2$ and $p^2_3=q^2_3$, which lead to $p_1=q_1$, $p_2=q_2$ and $p_3=q_3$.

Second, assume that all the second terms in \eqref{pq1}, \eqref{pq2} and \eqref{pq3}  equal 0; that is 
$p_2q_3=q_2p_3$, $p_3q_1=q_3p_1$ and $p_1q_2=q_1p_3$.
Thus, we have $q_1/p_1=q_2/p_2=q_3/p_3$, which leads to $p_1=q_1$, $p_2=q_2$ and $p_3=q_3$.
We observe that one of the three equations $p_2q_3=q_2p_3$, $p_3q_1=q_3p_1$ and $p_1q_2=q_1p_3$ is redundant.
Hence, if two of the second terms in \eqref{pq1}, \eqref{pq2} and \eqref{pq3}  equal 0, the third one also equals 0.
 
In the final step,  we consider the case that two of the first terms in  \eqref{pq1}, \eqref{pq2} and \eqref{pq3}  equal 0. Without loss of generality, 
we assume $q_1q_3p_2=p_1p_3q_2$,
$q_2q_1p_3=p_2p_1q_3$ and  $p_1q_2=q_1p_2$.
By  $q_1q_3p_2=p_1p_3q_2$ and $p_1q_2=q_1p_2$,  we get $p_3=q_3$. Combining with $q_2q_1p_3=p_2p_1q_3$ , we have $p_1^2=q_1^2$ and $p_2^2=q_2^2$, which gives $p_1=q_1$ and $p_2=q_2$.

In conclusion, if $(X,Y) \in \mathrm{IC}_r$  with $r\neq 0$, then $\mathbf{p}=\mathbf{q}$, or equivalently, $F_X=F_Y$. 
\end{proof}

%
\begin{proof}[\textbf{\upshape Proof of Theorem \ref{thm:main:sec:3}}]
The ``if" part is clear. For the ``only if" part, we will show that if $(X,Y) \in \mathrm{IC}$, then $\corr(X,Y)=0$.  
We first discuss the case $X$ and $Y$ have different supports.  Let $A=\supp(X)$ and $B=\supp(Y)$.
\begin{enumerate}[(a)] 
 \item Suppose $A\cap B \neq A$ and $A \cap B \neq B$. In this case, we can find a function $g$ such that $g(X)$ and $g(Y)$ are  bi-atomic random variables with different supports.  By Proposition~\ref{lem:h} and Proposition \ref{pro:bi-atomic}, we have that $(g(X),g(Y))\in \mathrm{IC}$ and $g(X)$ and $g(Y)$ are independent, which gives $\corr(g(X), g(Y))=0$.  As $(X,Y) \in \mathrm{IC}$, we have $\corr(X,Y)=0$.

 \item Suppose  $A \subsetneq B$. We can find a function such that
         $g(X)$ is a bi-atomic random variable and $g(Y)$ is a tri-atomic random variable. By Proposition~\ref{lem:h} and Lemma \ref{pro:tri}, we have  that $(g(X),g(Y))\in \mathrm{IC}$ and $\corr(g(X), g(Y))=0$. Hence, $\corr(X,Y)=0$. If $B \subsetneq A$, we can also have $\corr(X,Y)=0$.
 \end{enumerate}
 Next, we discuss the case $X$ and $Y$ have the same support.
 Let $S=\supp(X)=\supp(Y)$. 
 Because $X$ and $Y$ have different distributions, there exists a set $A \subsetneq S$ such that $\p(X \in A) \neq \p(Y \in A)$.
 Furthermore, we can also find $B \subsetneq S$ such that $B\cap A=\emptyset$ and $A \cup B \subsetneq S$  as $S$ contains more than two points. Thus, the sets $A$, $B$ and $S/(A \cup B)$ are three non-empty and exclusive sets. Let $g(x)=\id_A+2\id_B+3\id_{S/(A \cup B)}$.
  It is clear that $g(X)$ and $g(Y)$ are tri-atomic random variables with the same support $[3]$ but have different distributions. Therefore, by Lemma \ref{pro:tri}, we have $\corr(g(X), g(Y))=0$, and thus  $\corr(X, Y)=0$.
\end{proof}

\subsection{Proofs in Section \ref{sec:identical}}
\label{app:A4}

\begin{proof}[\textbf{\upshape Proof of Proposition \ref{prop:ic:finite:identical}}]
We first show the necessity.
If $(X,Y)\in \mathrm{IC}_r$, then
\begin{align}\label{eq:ic:finite}
\corr(g(X),g(Y))=\frac{\bz\T(P-\bp\bp\T)\bz}{\bz\T(D-\bp\bp\T)\bz}=r,\quad  \bz=g(\bx)=(g(x_1),\dots g(x_n))^\top,
\end{align}
holds for all admissible function $g$, that is, for all real vectors $\bz \in \breve \R^n$ with $\breve{\R}^n = \R^n\backslash\{c \bone_n:c \in \R\}$.
Equation~\eqref{eq:ic:finite} is equivalent to
\begin{align}\label{eq:ic:finite:quadratic}
\bz\T(P - r D - (1-r)\bp\bp\T)\bz=0\quad \text{for all } \bz \in \breve \R^n.
\end{align}
This implies that
$(P +P\T-2 r D - 2(1-r)\bp\bp\T)\bz=\bzero_n$ for all $\bz \in \breve \R^n$.
By taking $\bz = \be_i$ for $i\in [n]$, we have that $P +P\T-2 r D - 2(1-r)\bp\bp\T=O$, and thus we obtain~\eqref{eq:ic:finite:mixture:m:pi}.

Next, we show the sufficiency.
Suppose that the probability matrix $P$ satisfies~\eqref{eq:ic:finite:mixture:m:pi}.
Then we have that
$$
P-rD-(1-r)\bp\bp\T + P\T-rD-(1-r)\bp\bp\T=O.
$$
Therefore, for every $\bz \in \R^n$, we have
\begin{align*}
\bz\T(P-rD-(1-r)\bp\bp\T)\bz& = \bz\T(P-rD-(1-r)\bp\bp\T)\T\bz 
= \bz\T(P\T-rD-(1-r)\bp\bp\T)\bz \\
&= - \bz\T(P-rD-(1-r)\bp\bp\T)\bz,
\end{align*}
and thus $ \bz\T(P-rD-(1-r)\bp\bp\T)\bz=0$.
Therefore, we have $(X,Y)\in \mathrm{IC}_r$.

In order for $r D + (1-r)\bp\bp\T$ to be a probability matrix, $r\in[-1,1]$ has to satisfy
$0\leq rp_j + (1-r)p_j^2 \leq 1$ for all $j\in [n]$ and $0 \leq (1-r)p_ip_j\leq 1$ for all $i,j \in [n],~ i\neq j$,
Equivalently, $-p_j/(1-p_j)\leq r \leq 1+1/p_j$ for all  $j\in [n]$, and $1-1/(p_ip_j)\leq r$ for all $i,j \in [n],~i\neq j$.
Therefore, $r$ satisfies \eqref{eq:range:invariant:correlation}.
\end{proof}

\begin{proof}[\textbf{\upshape Proof of Theorem~\ref{thm:main:sec:4}}]

\noindent
We first show the necessity.
It suffices to show~\eqref{eq:identical:identity} for every $x,y\in \supp(X)$ since the cumulative distribution functions change only at such points.

The case that $X$ and $Y$ are bi-atomic is verified in Proposition \ref{prop:ic:finite:identical}. In what follows, we 
suppose $X, Y \in \mathcal{L}^2$  are not bi-atomic.
Let $x$ and $y$ be any two distinct points in $\supp(X)$.
In view of Proposition~\ref{prop:ic:finite:identical}, it suffices to consider the case when there exist three distinct points $z_1,z_2,z_3\in \supp(X)$, $z_1 < z_2 < z_3$, with two of them equal $x$ and $y$.
Let $z_0 = -\infty$ and $z_4 = \infty$.
Define
\begin{align}\label{eq:h:function}
h(t)=\sum_{i=0}^3 \id_{\{ t > z_{i}\}},\quad  t \in \R.
\end{align}
Since $h$ is admissible, Proposition~\ref{lem:h} implies that $(h(X),h(Y))\in \mathrm{IC}_r$.
Note that $h(X)$ and $h(Y)$ are identical random variables taking the value $i$ with probability $p_i=\p(z_{i-1}<X\le z_i)$ for $i\in[4]$, respectively.
Since $z_1,z_2,z_3 \in \supp(X)$, we have $p_1,p_2,p_3>0$.
We exclude the index $i=4$ in the following proof when $p_4=0$.
The joint probability matrix of $(h(X),h(Y))$, denoted by $P=(p_{ij})_{4\times 4}$, is then given by
\begin{align*}
p_{ij} =\p\left(h(X)=i,h(Y)=j\right) &=
\p(z_{i-1}< X \le z_i,\,z_{j-1} < Y \le z_j)\\
&=H(z_i,z_j)-H(z_{i-1},z_j)-H(z_{i},z_{j-1})+H(z_{i-1},z_{j-1}) 
 =:p_{ij}(H)
\end{align*}
for $i,j \in [4]$.
Let $H\T(x,y)=H(y,x)$ be the distribution function of $(Y,X)$.
By Proposition~\ref{prop:ic:finite:identical},
we have that
\begin{align}\label{eq:identity:reduced}
\frac{p_{ij}(H)+p_{ij}(H\T)}{2}=r p_i\id_{\{i=j\}}+ (1-r)p_i p_j,\quad  i,j \in \{1,2,3,4\}.
\end{align}
Since $p_{11}(H)=p_{11}(H\T)=H(z_1,z_1)$ and $p_1 = F(z_1)$, we have
$H(z_1,z_1)=r \min(F(z_1),F(z_1))+ (1-r)F(z_1)F(z_1)$.
Next, by taking $(i,j)=(1,2)$ in \eqref{eq:identity:reduced}, we have
\begin{align*}
\frac{H(z_1,z_2)+H(z_2,z_1)}{2}&=H(z_1,z_1) + \frac{p_{12}(H)+p_{12}(H\T)}{2}\\
&= r \min(F(z_1),F(z_1))+ (1-r)F(z_1)F(z_1)  + (1-r)F(z_1)\{F(z_2)-F(z_1)\}\\
&= r \min(F(z_1),F(z_2))+ (1-r)F(z_1)F(z_2).
\end{align*}
By repeating this calculation for all pairs of indices $(i,j) \in [4]^2$, we eventually obtain
\begin{align*}
\frac{H(z_i,z_j)+H(z_i,z_j)}{2}=r \min(F(z_i),F(z_j))+ (1-r)F(z_i)F(z_j),
\end{align*}
for all $(i,j)\in[4]^2$.
Since the set $\{z_1,z_2,z_3\}$ includes $x$ and $y$, we have~\eqref{eq:identical:identity} for $(x,x)$, $(x,y)$, $(y,x)$ and $(y,y)$.
Since $x$ and $y$, $x\neq y$ are taken arbitrary in $\supp(X)$,
we obtain the desired identity~\eqref{eq:identical:identity}.


Next, we show the sufficiency.
We show $(X,Y) \in \mathrm{IC}_r$ when $(X,Y)\sim H$ satisfies~\eqref{eq:identical:identity}. 

Let $g$ be an admissible function.
Denote by $H_g$ and $F_g$ the joint and marginal distributions of $(g(X),g(Y))$, respectively.
By taking $A=\{t\in \R:g(t)\le x\}$ and $B=\{t\in \R:g(t)\le y\}$ in~\eqref{eq:identical:identity:set}, we have that
\begin{align}\label{eq:identity:Hg}
\frac{H_g(x,y)+H_g(y,x)}{2} = r \min(F_g(x),F_g(y)) + (1-r)F_g(x)F_g(y),\quad (x,y) \in \R.
\end{align}
Note that $(x,y)\mapsto H_g(y,x)$ is the distribution function of $(g(Y),g(X))$ and the correlation coefficient is invariant under permutation, that is, $\cov(g(X),g(Y))=\cov(g(Y),g(X))$.
Together with the Fr\'echet-Hoeffding identity~\citep[Lemma~2 of~][]{L66}, \eqref{eq:identity:Hg} implies that
{\begin{align*}
&\cov(g(X),g(Y))+\cov(g(X),g(Y))\\
&=2\int_{\R}\int_{\R}
\left\{ r \min(F_g(x),F_g(y)) + (1-r)F_g(x)F_g(y) -F_g(x)F_g(y)
\right\}
\,\mathrm{d}x\,\mathrm{d}y\\
&= 2r\int_{\R}\int_{\R}
\left\{ \min(F_g(x),F_g(y)) -F_g(x)F_g(y)
\right\}
\,\mathrm{d}x\,\mathrm{d}y\\
&= 2r \var(g(X)) < \infty.
 \end{align*}}
Therefore, we have $\corr(g(X),g(Y))=r$ and we conclude $(X,Y)\in \mathrm{IC}^{\uparrow}_r$.
\end{proof}


\begin{proof}[\textbf{\upshape Proof of Corollary \ref{cor:ic:cont:identical}}]
It suffices to show that the range of the admissible invariant correlation is $0\le r \leq 1$.
Suppose~\eqref{eq:ic:cont:mixture:m:pi} holds.
Since the LHS of~\eqref{eq:ic:cont:mixture:m:pi} is a copula, so is the RHS.
Therefore, we have that
$
V_{rM + (1-r)\Pi}([\underline a,\overline a]\times [\underline b,\overline b])\geq 0
$
for all $\underline a,\overline a,\underline b,\overline b\in\I$ such that $\underline a\leq \overline a$ and $\underline b\leq \overline b$, where $V_C$ is a $C$-volume of a copula $C$.
Consider the case when $\underline a\leq \underline b \leq \overline b\leq \overline a$.
Then, for $l:= \overline a - \underline a$, we have that $$V_{rM + (1-r)\Pi}([\underline a,\overline a]\times [\underline b,\overline b])=r ( \overline b - \underline b) + (1-r)  ( \overline a - \underline a) ( \overline b - \underline b)\ge 0,$$ that is, $r + (1-r)l\geq 0$. Equivalently, we have $r \geq - l / (1-l)$. By letting $l\downarrow 0$, we obtain $r \geq 0$.
 \end{proof}

\subsection{Proofs in Section \ref{sec:cha:ICM}}\label{sec:proof:cha:ICM}

\begin{proof}[\textbf{\upshape Proof of Proposition \ref{prop:newmodel}}]
 \begin{enumerate}[(i)]
 \item For $i\in [d]$ and $t\in [0,1]$, we have $$\p(X_i\le t) = \sum_{j=1}^k \p( Z_{ij} =1) \p(U_j\le t) = t  \sum_{j=1}^k \E[Z_{ij}]=t .$$ 
 Hence, $X_i$ has the standard uniform distribution. 
 \item We can check that the copula of $(X_1,X_2)$  is given by
\begin{align*} \p(X_1\le u_1,X_2\le u_2) & = \sum_{j=1}^k  \p(Z_{1j }Z_{2j}=1) \min(u_1,u_2) +\left(1- \sum_{j=1}^k \p(Z_{1j }Z_{2j}=1)   \right) u_1u_2
\\& = \E[\mathbf Z_1^\top \mathbf Z_2] \min(u_1,u_2) +\left(1-  \E[\mathbf Z_1^\top \mathbf Z_2] \right) u_1u_2,
\end{align*} 
and this argument also applies to the other pairs. 
Therefore, each pair of components of $\bX$ has a positive Fr\'echet copula, and hence it has an invariant correlation. By Proposition \ref{prop:high-d}, we know that $\bX $ has an invariant correlation matrix. 
 \item  From (ii), we know that the correlation between $X_1$ and $X_2$ is $\E[ \mathbf Z_1^\top \mathbf Z_2] $, and similarly for the other pairs. Therefore, the correlation matrix of  $\mathbf X$ is  $\E[\Gamma  \Gamma^\top]$. \qedhere 
 \end{enumerate} 
 \end{proof}

For the proof of  Theorem \ref{thm:ch:mat}, we first present the following two lemmas.

 \begin{lemma}\label{prop:z}
 For $k\in \N$, the following statements hold:
 (i) $\mathcal Z_{d,k}$ is increasing in $k$;
 (ii) $\mathcal Z_{d,k}$ is convex;
 (iii) $\mathcal Z_{d,k}\subseteq \Theta_d$. In particular, any matrix in $\mathcal Z_{d,k}$ is a correlation matrix.
 \end{lemma}
 \begin{proof}[\textbf{\upshape Proof of Lemma \ref{prop:z}}]
 For (i), it suffices to note   $\mathcal Z_{d,k-1}\subseteq  \mathcal Z_{d,k}$ for $k\ge 2$, which can be checked by taking $Z_{i k}=0$ for all $i\in [d]$  in \eqref{eq:newmodel}. 
 Next, (ii) can be checked by the fact that for any event $A$ independent of two $d\times k$ categorical random matrices $\Gamma$ and $\Gamma'$, the matrix $\id_A \Gamma +(1-\id_A) \Gamma'$ 
 is a $d\times k$ categorical random matrix.
 Finally, (iii) follows directly from Proposition \ref{prop:newmodel}. \qedhere
 \end{proof}

    \begin{lemma}\label{prop:compatibility:1}
  $\mathcal Z_{d,k}= \mathrm{Conv}(\mathcal S_{d,k})$. 
  \end{lemma}
  \begin{proof}[\textbf{\upshape Proof of Lemma \ref{prop:compatibility:1}}]
   We have that
   $\mathcal Z_{d,k}\supseteq \mathrm{Conv}(\mathcal S_{d,k})$ since $\mathcal Z_{d,k}\supseteq \mathcal S_{d,k}$ and the set $\mathcal Z_{d,k}$ is convex.
   Therefore, it suffices to prove that $\mathcal Z_{d,k}\subseteq \mathrm{Conv}(\mathcal S_{d,k})$.

Let $\Gamma$ be a $d\times k$ categorical random matrix such that $\mathbb{E}[\Gamma \Gamma^\top] \in \mathcal Z_{d,k}$.
Let $\gamma=(\boldsymbol{\gamma}_1,\dots,\boldsymbol{\gamma}_d)^\top\in \{0,1\}^{d\times k}$ be any realization of the random matrix $\Gamma$.
Define $A(\gamma)= (A_1(\gamma),\dots,A_k(\gamma))$, where $A_s(\gamma)=\{i\in[d]:\gamma_{is}=1\}$, $s \in [k].$
Then $A(\gamma)$ is a partition of $[d]$ into $k$ subsets since
$\sum_{s=1}^k \gamma_{is}=1$ and thus every $i \in [d]$ belongs to only one of the partitioned subsets.
Moreover, we have $\gamma^\top\gamma\in \mathcal S_{d,k}$ since
\begin{align*}
(\gamma^\top\gamma)_{ij}=\boldsymbol{\gamma}_i^\top \boldsymbol{\gamma}_j=\sum_{s=1}^k\gamma_{is}\gamma_{js}=\sum_{s=1}^k\id_{\{\gamma_{is}=1\,\gamma_{js}=1\}}
=\sum_{s=1}^k\id_{\{i,j\in A_s(\gamma)\}}.
\end{align*}
Therefore, we have that
$\mathbb{E}[\Gamma\Gamma^\top]=\sum_{\gamma} \gamma\gamma^\top \p(\Gamma=\gamma)\in \mathrm{Conv}(\mathcal S_{d,k}).
$
As a result, we have $\mathcal Z_{d,k}=\mathrm{Conv}(\mathcal S_{d,k})$.
  \end{proof}

\begin{proof}[\textbf{\upshape Proof of Theorem \ref{thm:ch:mat}}]
The last two equalities follow directly from Part~(i) of Lemma~\ref{prop:z}, Lemma~\ref{prop:compatibility:1} and the fact that $\mathcal S_{d, k}=\mathcal S_{d,d}$ for $k\ge d$.
Moreover, Proposition~\ref{prop:newmodel} implies that $\Theta_d \supseteq \mathcal Z_{d,d}$.
  Therefore, it suffices to prove $\Theta_d \subseteq  \mathrm{Conv}(\mathcal S_{d,d})$.
  
  Let $\mathbf{Y}=(Y_1,\dots,Y_d)^\top$ be a continuous random vector with standard uniform marginals and invariant correlation $R=(r_{ij})_{d\times d}$.
  Let $L=\{(u,v)\in [0,1]:u<v\}$. 
  Then its volume with respect to $M$ and $\Pi$ are $V_M(L)=0$ and $V_{\Pi}(L)=1/2$, respectively.  
  By  Proposition~\ref{prop:high-d}, we have that
  \begin{align*}
\p(Y_i\neq Y_j)=\p((Y_i,Y_j)\in L)+\p((Y_j,Y_i)\in L)=2r_{ij}V_M(L)+2(1-r_{ij})V_{\Pi}(L)
=1-r_{ij},
  \end{align*}
  and hence $\p(Y_i=Y_j)=r_{ij}$ for all $i,j\in [d]$.

For $\mathbf{y}\in [0,1]^d$, let $A(\mathbf{y})=(A_1(\mathbf{y}),\dots,A_d(\mathbf{y}))$ be any partition of $[d]$ with $d$ sets such that two indices $i$ and $j$ are in the same partitioned subset if and only if $y_i=y_j$.
In addition, define the matrix $\Sigma(\mathbf{y})=(\Sigma_{ij}(\mathbf{y}))_{d\times d}$ by $\Sigma_{ij}(\mathbf{y})=\sum_{s=1}^d \id_{\{i,j \in A_s(\mathbf{y})\}}$.
Note that, although $\mathbf{y}$ takes a value on the uncountable set $[0,1]^d$, there is only a finite number of different partitions of $[d]$, and thus so is $\Sigma(\mathbf{y})$.
Let $\Sigma^{(1)},\dots,\Sigma^{(N)}$, $N\in \mathbb{N}$, be all possible such matrices.
For the given random vector $\mathbf{Y}$, let $\alpha_n=\p(\Sigma(\mathbf{Y})=\Sigma^{(n)})$, $n \in [N]$.
Then
\begin{align*}
r_{ij}=\p(Y_i=Y_j)=\mathbb{E}[\id_{\{Y_i=Y_j\}}]=\mathbb{E} \left[\sum_{s=1}^d\id_{\{i,j\in A_s(\mathbf{Y})\}}\right]=\mathbb{E}[\Sigma_{ij}(\mathbf{Y})].
\end{align*} 
Since $\Sigma^{(n)}\in \mathcal S_{d,d}$ for every $n\in [N]$ and $\sum_{n=1}^N \alpha_n=1$, we conclude that
$
R=\mathbb{E}[\Sigma(\mathbf{Y})]=\sum_{n=1}^N \alpha_n \Sigma^{(n)} \in \mathrm{Conv}(\mathcal S_{d,d}).
$
Therefore, we obtain the desired result.
  \end{proof}
    
\subsection{Proofs in Section \ref{sec:PRD}}
  
  \begin{proof}[\textbf{\upshape Proof of Proposition \ref{prop:X:prd}}]
Fix $i\in [d]$ and an increasing set $A\subseteq \mathbb{R}^d$. 
For every $s\in [k]$, the joint law of $(Z_{is},X_i)$ is identical to that of $(Z_{is},U_s)$, where $Z_{is}$ is the $(i,s)$th element of $\Gamma$. Indeed, for any $A\subseteq[0,1]$, we have that
\begin{align*}
    \p(Z_{is}=1,~X_i \in A)&=\p\left(Z_{is}=1,~\sum_{l=1}^k Z_{il}U_l \in A\right)=
    \p(Z_{is}=1,~U_s \in A).
\end{align*}
This implies the independence between $Z_{is}$ and $X_i$ since
\begin{align*}
    \p(Z_{is}=1,~X_i \in A)&=
    \p(Z_{is}=1,~U_s \in A)=\p(Z_{is}=1)\p(U_s \in A) =\p(Z_{is}=1)\p(X_i \in A).
\end{align*}
 Therefore, we have that, for $x\in [0,1]$,
 \begin{align*}
\p(\mathbf{X}\in A\mid X_i = x)&=\sum_{s=1}^k\p(\mathbf{X}\in A\mid X_i=x, Z_{is}=1)\p(Z_{is}=1\mid X_i=x) \\&=\sum_{s=1}^k\p(\mathbf{X}\in A\mid U_s=x, Z_{is}=1)\p(Z_{is}=1),
\end{align*}
and hence it suffices to prove that $x\mapsto \p(\mathbf{X}\in A\mid U_s=x, Z_{is}=1)$ is increasing for every $s\in [k]$.
Let $\gamma_{-s}$ be a $d\times (k-1)$ matrix obtained from deleting the $s$th column from $\gamma$, and $\mathbf{u}_{-s}$ be a $(k-1)$-dimensional vector obtained from deleting the $s$th element from $\mathbf{u}$.
Since $U_s$ is independent of $\Gamma$ and $\mathbf{U}_{-s}$, we have that
\begin{align*}
\p(\mathbf{X}\in A\mid U_s=x, Z_{is}=1)&=\p(\Gamma_{-s}\mathbf{U}_{-s}+x(Z_{is},\dots,Z_{ds})^\top \in A\mid U_s=x,Z_{is}=1)\\
&=\p(\Gamma_{-s}\mathbf{U}_{-s}+x(Z_{is},\dots,Z_{ds})^\top \in A\mid Z_{is}=1).
\end{align*}
The probability in the last expression is increasing in $x$ since $(Z_{is},\dots,Z_{ds})\ge \mathbf{0}_{d}$ and $A$ is an increasing set.
\end{proof}

\subsection{Proofs in Section \ref{sec:ic:inc}}



We first present some lemmas for the proof of Theorem \ref{thm:icincrease0}. 

\begin{lemma}\label{lem:bi_tri}
Suppose that $X$ is bi-atomic and $Y$ is tri-atomic. If $(X,Y) \in \mathrm{IC}_r^{\uparrow}$ for some $r\in[-1,1]$, then $r=0$.
\end{lemma}
\begin{proof}[\textbf{\upshape Proof of Lemma \ref{lem:bi_tri}}]

Assume $\supp(X)=\{x_1, x_2\}$ for some $x_1<x_2$ and $\supp(Y)=\{y_1, y_2, y_3\}$ for some $y_1<y_2<y_3$. Let $P=(p_{ij})_{2\times 3}$ be the probability matrix, $\mathbf p=(p_1, p_2)^\top$ and $\mathbf q=(q_1, q_2, q_3)^\top$ be the marginal distributions of $X$ and $Y$, respectively, with $p_i=\p(X=x_i)$ and $q_j=\p(Y=y_j)$ for $i\in [2]$ and $j\in[3]$ and $S=(s_{ij})_{2\times 3}=P-\mathbf p^\top \mathbf q$.  We have $p_i>0$ for $i=1,2$ and $q_j>0$ for $j\in[3]$.

Since a linear transform of $g$ does not change $\corr(g(X), g(Y ))$, we can fix $g(x_1)=0$ and $g(x_2)=1$. 
Assume $g(y_1)=z_1$, $g(y_2)=z_2$ and $g(y_3)=z_3$ with $z_1\le z_2\le z_3$. As $(X,Y) \in \mathrm{IC}_r^{\uparrow}$, we have
$$\corr(g(X),g(Y))=\frac{z_1s_{21}+z_2s_{22}+z_3s_{23}}{\sqrt{p_2-p_2^2}\sqrt{z_1^2q_1+z_2^2q_2+z_3^2q_3-(z_1q_1+z_2q_2+z_3q_3)^2}}=r.$$
By matching the coefficients in front of $z_1^2$, $z_2^2$, $z_3^2$, $z_1z_2$, $z_1z_3$ and $z_2z_3$ terms, we get the system
\begin{subequations}
\begin{align}
&s_{21}^2=r^2(p_2-p_2^2)(q_1-q_1^2),\label{eq:ina}\\
&s_{22}^2=r^2(p_2-p_2^2)(q_2-q_2^2),\label{eq:inb}\\
&s_{32}^2=r^2(p_2-p_2^2)(q_3-q_3^2),\label{eq:inc}\\
&s_{21}s_{22}=r^2(p_2-p_2^2)q_1q_2,\label{eq:ind}\\
&s_{21}s_{23}=r^2(p_2-p_2^2)q_1q_3,\label{eq:ine}\\
&s_{22}s_{23}=r^2(p_2-p_2^2)q_2q_3.\label{eq:inf}
\end{align}
\end{subequations}
As $g$ can be any increasing function, at most one of $z_1,z_2,z_3$ can be fixed as 0.
If $z_3=0$, we have \eqref{eq:ina}, \eqref{eq:inb} and \eqref{eq:ind} hold simultaneously. Thus, $r\neq0$ implies $q_3=0$. 
If $z_2=0$, we have \eqref{eq:ina}, \eqref{eq:inc} and \eqref{eq:ine} hold simultaneously. Thus, $r\neq0$ implies $q_2=0$. 
If $z_1=0$, we have  \eqref{eq:inb}, \eqref{eq:inc} and \eqref{eq:inf} hold simultaneously. Thus, $r\neq0$ implies $q_1=0$. 
If none of $z_1,z_2,z_3$ is $0$, then we have $q_1=q_2=q_3 = 0$.
As $q_1, q_2, q_3>0$, we have $r=0$.
\end{proof}

\begin{lemma}\label{lem:tri}
Suppose both $X$ and $Y$ are  tri-atomic random variables with the same support.
If $(X,Y) \in \mathrm{IC}_r^{\uparrow}$ for some $r\in[-1,1]$, then $r=0$ or $F_X=F_Y$. \end{lemma}
\begin{proof}[\textbf{\upshape Proof of Lemma \ref{lem:tri}}]
As $p_i=\p(X=x_i)>0$ and $q_i=\p(Y=y_j)>0$ for all $i,j\in[3]$, we can use the similar argument in the proof of Lemma \ref{pro:tri} to show that if $(X,Y) \in \mathrm{IC}_r^{\uparrow}$ with $r\neq 0$, then $F_X=F_Y$. 
\end{proof}

\begin{lemma}\label{lem:n-atomic_inc}
Suppose $X$ is $m$-atomic and $Y$ is $n$-atomic with $m\ge 2$, $n>2$ and $\supp(X)\neq \supp(Y)$.  If $(X,Y) \in \mathrm{IC}_r^{\uparrow}$ for some $r\in[-1,1]$, then $r=0$.
\end{lemma}
\begin{proof}[\textbf{\upshape Proof of Lemma \ref{lem:n-atomic_inc}}]

It suffices to show that, for such $X$ and $Y$, there exists an increasing function $g$ such that $g(X)$ is bi-atomic and $g(Y)$ is tri-atomic, or $g(X)$ is tri-atomic and $g(Y)$ is bi-atomic. Then, we can use Lemma \ref{lem:bi_tri} to get  $r=0$.
The existence of $g$ can be checked directly by exhausting all possibilities of the supports of $X$ and $Y$.
For instance, if there exists $y_0\in \supp(Y)$ but not in $\supp(X)$, and 
each of the events $\{X<y_0\}$, $\{X>y_0\}$, $\{Y<y_0\}$ and $\{Y>y_0\}$ has non-zero probability, 
then the function $g: y\mapsto \id_{\{y>y_0\}} + \id_{\{y\ge y_0\}} $
is sufficient. 
We omit the details of all other cases here.
\end{proof}

\begin{proof}[\textbf{\upshape Proof of Theorem \ref{thm:icincrease0}}]
It is clear that $\mathrm{IC}_r\subseteq \mathrm{IC}^{\uparrow}_r$ for all $r\in [-1,1]$, which implies the ``only if" part. We will show the ``if" part below.

\begin{enumerate}[(i)]
\item If $(X, Y) \in \mathrm{IC}^{\uparrow}_0$, then $\cov(g(X), g(Y))=0$ for all admissible decreasing functions $g$.
By taking $g(x)=\id_{\{x\le a\}}$, we have $\cov(\id_{\{X\le a\}},\id_{\{Y\le a\}})=0$ for every $a\in \R$.
For every fixed $a,b\in \R$, let $g(x)=\id_{\{x\le a\}}+\id_{\{x\le b\}}$. Since  $\cov(g(X), g(Y))=0$, $\cov(\id_{\{X\le a\}},\id_{\{Y\le a\}})=0$ and $\cov(\id_{\{X\le b\}},\id_{\{Y\le b\}})=0$,  we have $\cov\left(\id_{\{X\le a\}},\id_{\{Y\le b\}}\right)+\cov\left(\id_{\{X\le b\}},\id_{\{Y\le a\}}\right)=0.$ Hence $(X,Y)$ is quasi-independent, and thus $(X,Y)\in \mathrm{IC}_0$ by Theorem \ref{thm:ch:zero:ic}.

\item Assume $X$ and $Y$ have identical distributions.
By Theorem~\ref{thm:main:sec:4}, it only remains to show that~\eqref{eq:identical:identity} holds when $(X,Y)\in \mathrm{IC}^{\uparrow}_r$.

Notice that the necessity part of the proof of Theorem~\ref{thm:main:sec:4} directly applies with a few modifications since, for example, the function~\eqref{eq:h:function} is increasing and admissible, and thus Item~(\ref{item:h:inc}) in Corollary~\ref{cor:ic:inc:basic:properties} is used instead of Proposition~\ref{lem:h}.
By carefully checking the proof of Theorem~\ref{thm:main:sec:4}, we only need to show that Proposition~\ref{prop:ic:finite:identical} holds for the case of $\mathrm{IC}_r^{\uparrow}$.
Using $\mathrm{IC}^{\uparrow}_r \subseteq \mathrm{IC}_r$, we know that $n$-atomic identically distributed random variables $X$ and $Y$ satisfy $(X,Y)\in \mathrm{IC}_r^{\uparrow}$ if $r$ satisfies~\eqref{eq:range:invariant:correlation} and the probability matrix $P$ satisfies~\eqref{eq:ic:finite:mixture:m:pi}.
We show the remaining ``only if" part of this statement.

To this end, let $X$ and $Y$ be $n$-atomic identically distributed random variables taking values in $\mathcal X=\{x_1,\dots,x_n\}$, $n\in \N$, with $x_1<\cdots<x_n$. 
Write $P=(p_{ij})_{n\times n}$, $p_{ij}=\p(X=x_i,Y=x_j)$, with $\bp=P\bone_n$ and $D=\diag(\bp)$.

Suppose that $(X,Y)\in \mathrm{IC}_r^{\uparrow}$ and let $A:=P +P\T-2 r D - 2(1-r)\bp\bp\T$.
Since~\eqref{eq:ic:finite} holds for all admissible increasing functions $g$, \eqref{eq:ic:finite:quadratic} holds for all $\bz \in \breve{\R}^n$ such that $z_1<\cdots < z_n$.
Therefore, for such $\bz$'s, we have $A\bz = \bzero_n$.
For $\bz \in \breve{\R}^n$ such that $z_1<\cdots < z_n$, there exists a sufficiently small $\delta>0$ such that $\bz + \epsilon \be_j \in \breve{\R}^n$ is still increasing for all $\epsilon < \delta$ and $j\in [n]$.
Since $A\bz=A(\bz + \epsilon \be_j)=\bzero_n$, we obtain $A\be_j=\bzero_n$ for all $j\in [n]$, and thus $A=O$.
Therefore, we have~\eqref{eq:ic:finite:mixture:m:pi}.
The range condition~\eqref{eq:range:invariant:correlation} of $r$ necessarily holds for $r D + (1-r)\bp\bp\T$ to be a probability matrix.

\item For $X,Y\in \mathcal{L}^2$ with different distributions and $|\supp(Y)|>2$, we will show that $(X,Y) \in \mathrm{IC}_0^{\uparrow}$ implies $r=0$. 


 As $X$ and $Y$ have different distributions, there exists $(a,b]$ such that $\p(X\in (a,b])\neq \p(Y\in (a,b])$. Hence, we can find an increasing function $g$ such that $g(X)$ is $m$-atomic and $g(Y)$ is $n$-atomic for some  $m\ge 2$ and $n>2$ where $g(X)$ and $g(Y)$ have different distributions. Hence, without loss of generality, we can assume $X$ is $m$-atomic and  $Y$ is $n$-atomic for some  $m\ge 2$ and $n>2$ with different distributions.

If $\supp(X)\neq \supp(Y)$, we have $r=0$ by Lemma \ref{lem:n-atomic_inc}.
Assume $\supp(X)=\supp(Y)$.
As $X,Y$ have different distributions, we can always find $a\in \supp(X)$ such that $\p(X>a)\neq \p(Y>a)$. For such $a$, let $g(x)=\id_{\{x>a\}}+\id_{\{x>b\}}+\id_{\{x>c\}}$ for some  $b,c \in \supp(X)$ with $a<b<c$.  Thus, $g(X)$ and $g(Y)$ are tri-atomic random variables with the same support but different distribution.
By Lemma \ref{lem:tri}, we have $r=0$.


Therefore, we have $(X,Y) \in \mathrm{IC}^{\uparrow}_0$. Using Item (i) above, we have $(X,Y) \in \mathrm{IC}_0$.\qedhere
\end{enumerate}
\end{proof}

\subsection{Proofs in Appendix~\ref{app:auxiliary}}

\begin{proof}[\textbf{\upshape Proof of Proposition~\ref{prop:tdm}}]
Since the statement on $\kappa_g$-matrices is obvious, we will show that $\mathbf{X}$ has the tail-dependence matrix $R$.
Without loss of generality, we can assume that the identical marginal of $\mathbf{X}$ is the standard uniform distribution.
Fix $i,j\in [d]$.
Since $\lambda_{ij}=\lambda_{ji}$, Proposition~\ref{prop:high-d} leads to
\begin{align*}
\lambda_{ij}&=\frac{\lambda_{ij}+\lambda_{ji}}{2}=\lim_{u\downarrow 0}\frac{C_{ij}(u,u)+C_{ji}(u,u)}{2u}
=\lim_{u\downarrow 0}\frac{r_{ij}M(u,u)+(1-r_{ij})\Pi(u,u)}{u}
=r_{ij}.
\end{align*}
\end{proof}

\begin{proof}[\textbf{\upshape Proof of Proposition \ref{prop:conformal}}]
Since $S_{n+1},\dots,S_{n+d}$ are independent conditional on the null training sample $\mathcal D=\{S_i: i\in [n]\}$, we have  
    \begin{align*}
        \p\left(P_i=\frac{j_i}{n+1},i\in \mathcal N\right)&=
        \mathbb{E}_{\mathcal D}\left.\left[\p\left(P_i=\frac{j_i}{n+1},i\in \mathcal N\,\right\vert\, \mathcal D\right)\right]\\
       &= \mathbb{E}_{\mathcal D}\left.\left[\prod_{i\in \mathcal N}\p\left(P_i=\frac{j_i}{n+1}\,\right\vert\, \mathcal D\right)\right]\\& =  \mathbb{E}_{\mathcal D}\left[\prod_{i\in \mathcal N}
           \left(S_{(j_i)}-S_{(j_i-1)}\right)
           \right]
        \end{align*}
    for $ j_i \in [n+1]$, where $0=S_{(0)}<S_{(1)}<\cdots < S_{(n)}<S_{(n+1)}=1$ are the order statistics of the null training sample $(S_1\dots,S_n)$.
    Since the conformal p-values are independent of the distribution of scores, we can assume without loss of generality that $S_1,\dots,S_n$ are independently distributed of the standard uniform distribution.
    Let $T_j=S_{(j)}-S_{(j-1)}$, $j\in [n+1]$.    Then $(T_1,\dots,T_{n+1})$ follows the Dirichlet distribution with the parameter vector $\bone_{n+1}\in\mathbb{R}^{n+1}$.
    Therefore, 
    \begin{align*}
    \mathbb{E}_{\mathcal D}\left[\prod_{i\in \mathcal N}
           \left(S_{(j_i)}-S_{(j_i-1)}\right)
           \right]&=
               \mathbb{E}_{\mathcal D}\left[
            T_1^{N_1}\cdots T_{n+1}^{N_{n+1}}
\right]\\&=\frac{B(1+N_1,\dots,1+N_{n+1})}{B(\bone_{n+1})}\\
&=\frac{N_1!\cdots N_{n+1}!}{(n+m)(n+m-1)\cdots(n+1)},
    \end{align*}
    where $B$ is the $(n+1)$-dimensional Beta function.

    In particular, when $m=1$, we have $\p(P_i=j/(n+1))=1/(n+1)$, $j \in [n+1]$.
When $m=2$, we have
\begin{align*}
\p\left(P_i=\frac{j}{n+1},P_{i'}=\frac{j'}{n+1}\right)=\frac{1}{n+2}\cdot \frac{1}{n+1}\id_{\{j=j'\}}+ \left(1-\frac{1}{n+2}\right)\cdot \frac{1}{n+1}\cdot \frac{1}{n+1},
\end{align*}
which is the exchangeable case of the model~\eqref{eq:ic:infinite:mixture:m:pi} and thus has the invariant correlation $1/(n+2)$.
\end{proof}

\end{document}